\documentclass[12pt,reqno]{amsart} 
\usepackage{amsmath}\allowdisplaybreaks
\usepackage{amscd,amssymb,stmaryrd}
\usepackage{enumitem}
\usepackage{pgfplots} 
\pgfplotsset{compat=newest}
\usetikzlibrary{decorations.pathreplacing,angles,quotes}
\usetikzlibrary{bending}
\usepackage{amsthm}
\usepackage[colorlinks]{hyperref}
\usepackage[lite]{amsrefs}

\numberwithin{equation}{section}

\numberwithin{chap}{section}
\newtheorem{thm}{Theorem}
\numberwithin{thm}{section}
\newtheorem{prop}[thm]{Proposition}
\newtheorem{defn}[thm]{Definition}
\newtheorem{lem}[thm]{Lemma}
\newtheorem{cor}[thm]{Corollary}

\begin{document}

\pagestyle{myheadings} \thispagestyle{empty} \markright{}
\title{On Sets Containing an Affine Copy of Bounded Decreasing Sequences}	
\author{Tongou Yang}  

\maketitle

\begin{abstract}
   How small can a set be while containing many configurations? Following up on earlier work of Erd\H os and Kakutani \cite{MR0089886}, M\'ath\'e \cite{MR2822418} and Molter and Yavicoli \cite{Molter}, we address the question in two directions. On one hand, if a subset of the real numbers contains an affine copy of all bounded decreasing sequences, then we show that such subset must be somewhere dense. On the other hand, given a collection of convergent sequences with prescribed decay, there is a closed and nowhere dense subset of the reals that contains an affine copy of every sequence in that collection.
\end{abstract}


\section{Introduction}
Given sets $A,B\subseteq \mathbb R$, we say that $A$ contains the pattern $B$ if $A$ contains an affine copy of $B$, i.e. if there exist $\delta\neq 0$ and $t\in \mathbb R$ such that $t+\delta B\subseteq A$. Identification of patterns in sets is an active research area, and there are questions of many flavours:
\begin{enumerate}
    \item {\it Which types of patterns are guaranteed to exist in large sets?} For example, a classical consequence of the Lebesgue density theorem is that if $E\subseteq \mathbb R$ has positive Lebesgue measure, then it contains an affine copy of all finite sets. In sets of fractal dimensions, \L aba and Pramanik \cite{MR2545245} proved that if a fractal set $A$ supports a measure satisfying a Frostman's condition and has sufficiently large Fourier decay, then $A$ must contain a $3$-term arithmetic progression. Last, but not least, one of the most famous conjectures in this direction is the Erd\"os distance conjecture; there are many substantial results established by Bennett, Greenleaf, Iosevich, Liu, Palsson, Taylor, etc. See \cite{MR3518531}\cite{MR3365800}\cite{MR3420476} for more details.
    \item {\it Can there exist large sets avoiding prescribed patterns?} A famous conjecture in this direction is the Erd\H os similarity problem (see \cite{MR0429704}), which is stated as follows: for each infinite set $S\subseteq \mathbb R$, does there exist a measurable set $E$ with positive Lebesgue measure that does not contain any affine copy of $S$? There are partial results to this conjecture by Bourgain, Falconer, Kolountzakis, etc; see \cite{Bourgain}\cite{MR722418}\cite{MR1446560}.
    
    Apart from Erd\H os similarity conjecture, there are also lots of well-known results above large sets avoiding patterns. Keleti \cite{MR2431353} showed that for any set $A \subseteq \mathbb R$ of at least 3 elements there exists a set of Hausdorff dimension $1$ that contains no similar copy of $A$. In this direction, Shmerkin \cite{MR3658188} showed that there exists a set of Fourier dimension $1$ that contains no $3$-term arithmetic progression. In another direction, Fraser and Pramanik \cite{MR3785600} obtained a general result that there exists sets of large Hausdorff dimension and full Minkowski dimension that avoids all patterns prescribed by a large family of functions.
    \item {\it How small can a set be while containing many patterns?} This will be the main point of concern in this article.
\end{enumerate}    

\subsection{Literature Review}
    In 1955, Erd\H os and Kakutani \cite{MR0089886} proved that there is a perfect set $A\subseteq [0,1]$ of Lebesgue measure $0$ and Hausdorff dimension $1$ which satisfies the following property: for each $n\geq 1$, there is $\eta_n>0$ such that if $P\subseteq \mathbb R$ is a finite set with $\leq n$ elements and with diameter $<\eta_n$, then there is $t\in \mathbb R$ such that $P+t\subseteq A$. In particular, such perfect $A$ with Lebesgue measure $0$ and Hausdorff dimension $1$ contains an affine copy of every finite set. This result marked the beginning of the study of small sets containing many prescribed patterns. 
    
    In 2008, M\'ath\'e \cite{MR2822418} constructed a compact set $C$ with Hausdorff dimension $0$ that contains an affine copy of all finite sets. Actually, the set $C$ he constructed contains a translate of every set that he calls a ``slalom". One can show that for every finite set $F$, there is a slalom that contains an affine copy of $F$. Looking closer into his construction, he is even able to show that $C$ contains an affine copy of every infinite bounded decreasing sequence with sufficiently rapid decay.
    
    In 2016, Molter and Yavicoli \cite{Molter} proved the following result: given a (possibly uncountable) family $\mathcal F$ of continuous functions on $\mathbb R^N$ obeying mild regularity conditions, there is an $F_\sigma$-set $E\subseteq \mathbb R^N$ of Hausdorff dimension $0$ such that
    $$
    \bigcap_{i\in \Lambda}f_i^{-1}(E)\neq \varnothing
    $$
    for any countable subcollection $\{f_i:i\in \Lambda\}\subseteq \mathcal F$. In particular, choosing $N=1$ and $\mathcal F=\{f_t(x)=x+t|t\in \mathbb R\}$, they are able to construct an $F_\sigma$-set $A\subseteq \mathbb R$ with Hausdorff dimension $0$ such that the following holds: given any $\{\alpha_m\}\subseteq \mathbb R$, there is $t\in \mathbb R$ such that $t+\alpha_m\in A$ for all $m$. A simpler proof of this special case is included in the appendix of this article.
    
    However, neither the set $E$ constructed in \cite{Molter} nor its simplification in the appendix of this paper is closed. In fact, even if a set $E\subseteq \mathbb R$ obeys the following weaker assumption:
    \begin{equation}\label{condition}
        \text{
        Given any $\{\alpha_m\}\subseteq \mathbb R$, there is $t\in \mathbb R$ and $\delta\neq 0$ such that $t+\delta\alpha_m\in E$ for all $m$,
        }
    \end{equation}
    then $\overline E$ should contain an interval. This can be seen by taking $S=\{\alpha_m\}$ to be an enumeration of all rationals in $[0,1]$. By assumption, there is $t\in \mathbb R$ and $\delta\neq 0$ such that $t+\delta S\subseteq E$. Taking closure on both sides shows that $[t,t+\delta]\subseteq \overline E$ if $\delta>0$ or $[t+\delta,t]\subseteq \overline E$ if $\delta<0$. If $E$ were closed, then $E$ itself should contain an interval, which would be a contradiction to the fact that $\mathrm{dim}_H(E)=0$. Thus, although $E$ in \cite{Molter} is small in terms of Hausdorff dimensions, it is quite large in the sense of topology.
    
    \subsection{Our main result}
    In this paper, we adopt a slightly different perspective from dimensionality which was the main concern of \cite{MR2822418} and \cite{Molter}. Instead, we use the topological notion of density to quantify largeness. A set is said to be {\it somewhere dense} if its closure contains an interval. We have just shown that any set $E$ satisfying Condition \eqref{condition} is somewhere dense; thus no closed set $E$ with $\mathrm{dim}_H(E)=0$ and satisfying Condition \eqref{condition} could be found. 
    
    As the simple example $\{\alpha_m\}=\mathbb Q\cap[0,1]$ suggests, the triviality of the problem above is mainly because $\{\alpha_m\}$ may have many accumulation points. Hence we weaken Condition \ref{condition} to the following:
    \begin{align}\label{condition2}
        &\text{
        Given any $\{\alpha_m\} $ which is strictly decreasing and bounded below, there is $t\in \mathbb R$ and}\nonumber\\
        &\text{$\delta\neq 0$ such that $t+\delta\alpha_m\in E$ for all $m$.
        }
    \end{align}
    Note that a bounded decreasing sequence has one and only one accumulation point in $\mathbb R$. This gives rise to the main question we are concerned in this paper.

{\bf Main question:} Let $E\subseteq \mathbb R$ be a set satisfying Condition \eqref{condition2}. Must $E$ be somewhere dense?

The answer to the main question is affirmative. This is the content of Theorem \ref{thm1} below.
 
\begin{thm}\label{thm1}
Let $E\subseteq \mathbb R$ be a set such that Condition \eqref{condition2} holds, i.e. for all sequences $\{\alpha_m\}_{m=1}^\infty$ strictly decreasing to $0$, there is $t\in \mathbb R$ and $\delta\neq 0$ such that $t+\delta\alpha_m\in E$ for all $m$. Then $E$ is somewhere dense.
\end{thm}

As we shall see, the proof of Theorem \ref{thm1} relies on the arbitrarily slow decay of $\{\alpha_m\}$. Interestingly, our next main theorem shows that this is the only obstruction to having a nowhere dense set obeying Condition \ref{condition2}. In fact, if we specify a sequence with a prescribed decay, however slow, one can turn Theorem \ref{thm1} into a negative result. In this case, we can even take such set $A$ to be closed and bounded. 

\begin{thm}\label{thm2}
	Let $\{\beta_m\}_{m=1}^\infty$ with $\beta_m\searrow 0$ strictly. Then there is a closed and nowhere dense set $A\subseteq [0,1]$, depending on $\{\beta_m\}$, such that for any sequence $\alpha_m\to 0$ with $|\alpha_m|=O(\beta_m)$, there is $\delta>0$ and $t\in \mathbb R$ such that $t+\delta \alpha_m\in A$ for all $m$.
\end{thm}
For example, we can take $\beta_m\searrow 0$ to be $(\log m)^{-1}$, or even $(\log\log m)^{-1}$, $(\log \log \log m)^{-1}$, etc, or we could take $\beta_m$ to be a fixed sequence that decreases slower than any finite iterations of the logarithmic function. Then we have the following corollary:
\begin{cor}
There is a closed, nowhere dense set $A\subseteq [0,1]$ containing an affine copy of all geometrically decreasing sequences (i.e. $\alpha_m=O(r^m)$ for some $0<r<1$), all sequences with polynomial decay (i.e. $\alpha_m=O(m^{-s})$ for some $s>0$) and all sequences with rate of decay faster than finitely many iterates of the logarithmic function (for example, $\alpha_m=O((\log (\log (\log m)))^{-1})$.
\end{cor}

\subsection{Generalisation to higher dimensions}
We may also consider extending Theorems \ref{thm1} and \ref{thm2} to higher dimensions. Let us take $n=2$ as an example.
\subsubsection{Extension of Theorem \ref{thm1} to higher dimensions}
Let $E\subseteq \mathbb R^2$. Suppose for all sequences $\alpha_m\to 0$ in $\mathbb R^2$, there is $\delta\neq 0$ and $t\in \mathbb R^2$ such that $t+\delta \alpha_m\in E$ for all $E$. Then what can we say about the density of $E$?

For example, given any unit vector $v\in \mathbb R^2$, we can take $\alpha_m$ to be any sequence converging to $0$ through the line with direction $v$. Applying Theorem \ref{thm1} on each line, we see that $E$ contains a line segment (although not necessarily of unit length) in every direction, so by another result of Keleti \cite{MR3451230}, $E$ must have the same Hausdorff dimension as a Kakeya set in $\mathbb R^2$. A famous result by Davies \cite{MR0272988} shows that any Kakeya set in $\mathbb R^2$ has full Hausdorff dimension, so $\mathrm{dim}_H(E)=2$.

But can we say more about $E$? Is it true that $\overline E$ contains an open ball in $\mathbb R^2$ as well? Since our proof of Theorem \ref{thm1} relies heavily on the interval structure on $\mathbb R$, it is not immediate to generalise the argument to the planar case.

\subsubsection{Extension of Theorem \ref{thm2} to higher dimensions}
Similarly, we may ask if given a sequence $\beta_m\searrow 0$ in $\mathbb R$ with prescribed decay, there is a closed and nowhere dense set $A\subseteq [0,1]\times [0,1]$ that contains an affine copy of every sequence $\mathbb R^2\ni\alpha_m\to 0$ with $|\alpha_m|=O(\beta_m)$. We hope to address this question in a future paper.

\subsection{Outline of the article}
This article has two main parts. The first part, from Section 2 to Section 5, gives the proof of Theorem \ref{thm1}. The second part is the proof of Theorem \ref{thm2}, which is given in Section 6. In the appendix we give a simple proof of the special case of Molter and Yavicoli's construction mentioned in the introduction, which is logically unrelated to the main theorems.

In Section 2, we introduce the necessary notation for this paper and do a preliminary reduction. In Section 3, we give a Cantor-like construction which is the key to the proof of Theorem \ref{thm1}. In Section 4, we construct a slowly decreasing sequence and use this, together with Lemma \ref{emp}, to prove Theorem \ref{thm1}. In Section 5, we prove Lemma \ref{emp}.

\section{Notation and reduction}
We start with some elementary lemmas in set theory and real number theory.
\subsection{Some set manipulations}
The following lemma on set relations will be used extensively in the article.
\begin{lem}\label{trans}
Let $A\subseteq \mathbb R$, let $\{A_i\subseteq \mathbb R:i\in I\}$ where $I$ is any index set, and let $t\in \mathbb R$. Then we have the following set relations:
\begin{align}
    (A+t)^c&=A^c+t,\label{ct}\\
    \bigcup_{i\in I}(A_i+t)&=\left(\bigcup_{i\in I}A_i\right)+t,\label{cupt}\\
    \bigcap_{i\in I}(A_i+t)&=\left(\bigcap_{i\in I}A_i\right)+t.\label{capt}
\end{align}
Hence without ambiguity, we may drop the parentheses in both sides of \eqref{cupt} and \eqref{capt}.
\end{lem}
\begin{proof}

    For \eqref{ct}, $x\in (A+t)^c$ if and only if $x\notin A+t$, if and only if $x-t\notin A$, if and only if $x-t\in A^c$, if and only if $x\in A^c+t$.
    
    For \eqref{cupt}, $x\in \cup_{i\in I} A_i+t$ if and only if there is $i\in I$ such that $x\in A_i+t$, if and only if there is $i\in I$ such that $x-t\in A_i$, if and only if $x-t\in \cup_{i\in I}A_i$, if and only if $x\in (\cup_{i\in I}A_i)+t$.
    
    For \eqref{capt}, $x\in \cap_{i\in I} A_i+t$ if and only if for all $i\in I$ we have $x\in A_i+t$, if and only if for all $i\in I$ we have $x-t\in A_i$, if and only if $x-t\in \cap_{i\in I}A_i$, if and only if $x\in (\cap_{i\in I}A_i)+t$.

\end{proof}

\subsection{Density of sets}
We adopt the following notation.
\begin{itemize}
    \item Given any interval $I\subseteq \mathbb R$, we use $|I|$ to denote its length. Any interval in this paper will be nondegenerate, that is, $|I|>0$.
    \item Given any set $S\subseteq \mathbb R$, we use $\overline{S}$ to denote its closure and $\mathrm{Int}(S)$ to denote its interior, both with respect to the standard topology on $\mathbb R$.
    \item If $K\subseteq \mathbb R$ is a closed interval, we say a set $S\subseteq K$ is dense in $K$ if for each open interval $I\subseteq K$ we have $I\cap K\neq \varnothing$. Equivalently, $S\subseteq K$ is dense in $K$ if $\overline{S}=K$.
    \item We say a set $A\subseteq \mathbb R$ is nowhere dense if $\mathrm{Int}(\overline{A})=\varnothing$. We say a set is somewhere dense if its closure contains an interval.
\end{itemize}  
We state the following lemma in real number theory.
\begin{lem}\label{lem1}
    The followings are equivalent.
    \begin{enumerate}
        \item \label{dense1} $A\subseteq \mathbb R$ is nowhere dense.
        \item \label{dense2} For each closed interval $K\subseteq \mathbb R$ there is an open subinterval $I\subseteq K$ such that $I\subseteq A^c$.
        \item \label{dense3} $A$ is not somewhere dense, that is, $\overline A$ contains no interval.
    \end{enumerate}
    As a corollary, If $A$ and $B$ are nowhere dense, then so is $A\cup B$.
\end{lem}
\begin{proof}
\begin{itemize}
    \item {\it \eqref{dense1} implies \eqref{dense2}.} Let $A\subseteq \mathbb R$ be nowhere dense. Assume, towards contradiction, that there is a closed interval $K\subseteq \mathbb R$ such that for all open intervals $I\subseteq K$, we have $I\cap A\neq \varnothing$. Then $\varnothing\neq I\cap A=(I\cap K)\cap A=I\cap (A\cap K)$. Since $I\subseteq K$ is arbitrary, by definition of density, $A\cap K$ is dense in $K$. Hence $\mathrm{Int}(\overline{A})\supseteq \mathrm{Int}(\overline{A\cap K})=\mathrm{Int}(K)\neq \varnothing$, which is a contradiction. Hence for any closed interval $K\subseteq \mathbb R$ there is some open subinterval $I\subseteq K$ such that $I\subseteq A^c$.
    
    \item {\it \eqref{dense2} implies \eqref{dense3}.} Suppose for any closed interval $K\subseteq \mathbb R$ there is some open subinterval $I\subseteq K$ such that $I\subseteq A^c$. Suppose, towards contradiction, that $A$ is somewhere dense. Then $\overline A$ contains an interval, which in turn contains some closed interval $K$. By assumption, there is some open interval $I\subseteq K$ such that $I\subseteq A^c$. But $A\supseteq K$, so $A^c\subseteq K^c$, so $I\subseteq K^c$. But since $I\subseteq K$, this is a contradiction.
    
    \item {\it \eqref{dense3} implies \eqref{dense1}.} We prove the contrapositive, that is, assuming $A$ is not nowhere dense, we are going to prove that $A$ is somewhere dense. Since $A$ is not nowhere dense, we have $\mathrm{Int}(\overline{A})\neq \varnothing$. As $\mathrm{Int}(\overline{A})$ is an open set, it contains an open interval $I$. Thus $I\subseteq \mathrm{Int}(\overline{A})\subseteq \overline A$, so $A$ is somewhere dense.
\end{itemize}
Now we prove the corollary. Let $A$ and $B$ be nowhere dense. By equivalence of \eqref{dense1} and \eqref{dense2}, we will show that for any closed interval $K\subseteq \mathbb R$ there is an open interval $I\subseteq K$ such that $I\subseteq (A\cup B)^c$. Now given any closed interval $K\subseteq \mathbb R$. Since $A$ is nowhere dense, by equivalence of \eqref{dense1} and \eqref{dense2} again, there is an open interval $I'\subseteq K$ such that $I'\subseteq A^c$. But $I'$ contains some closed interval $K'$. Since $B$ is nowhere dense, applying \eqref{dense2} to $K'$ gives an open interval $I\subseteq K'$ such that $I\subseteq B^c$. But $K'\subseteq I'\subseteq A^c$, so $I\subseteq A^c$. Hence $I\subseteq K$ is an open interval such that $I\subseteq A^c\cap B^c=(A\cup B)^c$, so $A\cup B$ is nowhere dense.
\end{proof}

\subsection{Two useful notations}\label{2not}
For our future use, it is convenient to introduce the following notations:
\begin{itemize}
    \item If $I$ is an interval with endpoints $-\infty<a<b<\infty$, we define $I^*:=[a,b)$. If $O$ is a union of intervals $I_n$ with endpoints $-\infty<a_n<b_n<\infty$ such that $\overline{I}_n\cap \overline{I}_{n'}=\varnothing$ for $n\neq n'$, we further define $O^*:=\cup_{n}I_n^*=\cup_n [a_n,b_n)$. (Note that by the Lindel\"of property of $\mathbb R$, such union is necessarily countable or finite.)
    \item For any set $S\subseteq \mathbb R$ and any $r>0$, we write $B_-(S,r)$ for the left $r$-neighbourhood of the set $S$: $B_-(S,r):=\{x-t:x\in S,\,\,0\leq t< r\}$. 
\end{itemize}
We list here some elementary properties we shall use.
    \begin{prop}\label{Bminus}
    \begin{enumerate}[label=(\roman*)]
        \item \label{not1} If $S=(a,b)$, then for each $r>0$, $B_-(S,r)=(a-r,b)$. In particular, $B_-(S,r)\supseteq [a,b)=S^*$.
        \item \label{not3} For any index set $I$ and any $r>0$, $\cup_{i\in I}B_-(S_i,r)=B_-(\cup_{i\in I}S_i,r)$.
        \item \label{not1a} If $S$ is a (countable or finite) union of bounded open intervals with disjoint closures, then for any $r>0$, $B_-(S,r)\supseteq S^*$.
        \item \label{not2} If $S_2\supseteq S_1$, then for any $r>0$, $B_-(S_2,r)\supseteq B_-(S_1,r)$.
        \item \label{not4} If $r<s$, then for any set $S$, $B_-(S,r)\subseteq B_-(S,s)$.
    \end{enumerate}
    \end{prop}
\begin{proof}
\begin{enumerate}[label=(\roman*)]
    \item Let $S=(a,b)$ and $r>0$. If $y\in B_-(S,r)$, then there is $x\in S=(a,b)$ and $0\leq t<r$ such that $y=x-t$, so $y\in (a-t,b-t)\subseteq (a-r,b-0)=(a-r,b)$. Hence $B_-(S,r)\subseteq (a-r,b)$.
    
    On the other hand, if $y\in (a-r,b)$, then we have two cases:
    
         If $a<y<b$, then letting $x=y\in (a,b)$ and $t=0$ shows that $y\in B_-(S,r)$.
         
        If $a-r<y\leq a$, then we let $\delta=a-y\in [0,r)$, and let $0<\epsilon<\min\{b-a,r-\delta\}$. Then we let $x=a+\epsilon\in (a,b)=S$ and $t=x-y$. Note that $x-y>a-y\geq 0$ and $x-y=a+\epsilon-y=\delta+\epsilon<\delta+r-\delta=r$. Thus $t\in [0,r)$ and so $y=x-t\in B_-(S,r)$. 
        
        Hence $B_-(S,r)\supseteq (a-r,b)$. Combining two directions we get $B_-(S,r)= (a-r,b)$.
       
        Since $a-r<a$ for all $r>0$, we have $B_-(S,r)=(a-r,b)\supseteq [a,b)$.
    \item Let $\{S_i\}_{i\in I}$ and $r>0$. If $y\in\cup_{i\in I}B_-(S_i,r)$, then there is $i\in I$ such that $y\in B_-(S_i,r)$, that is, there is $x\in S_i$ and $0\leq t<r$ such that $y=x-t$. But $S_i\subseteq \cup_{i\in I}S_i$, so $x\in \cup_{i\in I}S_i$, and thus $y\in B_-(\cup_{i\in I}S_i,r)$. Hence $\cup_{i\in I}B_-(S_i,r)\subseteq B_-(\cup_{i\in I}S_i,r)$.
    
    On the other hand, if $y\in B_-(\cup_{i\in I}S_i,r)$, then there is $x\in \cup_{i\in I}S_i$ and $0\leq t<r$ such that $y=x-t$. Since $x\in \cup_{i\in I}S_i$, there is $i\in I$ such that $x\in S_i$. Hence $y=x-t\in B_-(S_i,r)\subseteq \cup_{i\in I}B_-(S_i,r)$. Hence $\cup_{i\in I}B_-(S_i,r)\supseteq B_-(\cup_{i\in I}S_i,r)$.

    \item Write $S=\cup_n (a_n,b_n)$. Then for each $r>0$, 
    $$
    B_-(S,r)\stackrel{\ref{not3}}{=}\bigcup_n B_-((a_n,b_n),r)\stackrel{\ref{not1}}{=}\bigcup_n (a_n-r,b_n)\stackrel{\ref{not1}}{\supseteq} \bigcup_n [a_n,b_n)=S^*.
    $$
    
    \item Since $S_2\supseteq S_1$ we can write $S_2=(S_2\backslash S_1)\cup S_1$. By \ref{not3} we have $B_-(S_2,r)=B_-(S_2\backslash S_1,r)\cup B_-(S_1,r)\supseteq B_-(S_1,r)$.
    
    \item Let $r<s$, and let $y\in B_-(S,r)$. Then there is $x\in S$ and $0\leq t<r$ such that $y=x-t$. But then $0\leq t<s$, so $y\in B_-(S,s)$. Hence $B_-(S,r)\subseteq B_-(S,s)$.
\end{enumerate}
\end{proof}

\subsection{A preliminary reduction}\label{reduction}
From the statement of Theorem \ref{thm1}, given any $\alpha_m\searrow 0$, there is $t\in \mathbb R$ and $\delta\neq 0$ such that $t+\delta \alpha_m\in E$ for all $m$. However, $\delta$ can be either positive or negative. In this subsection, we shall show that without loss of generality, it suffices to prove the case when $\delta>0$. More precisely, we consider the following condition, which is slightly stronger than Condition \eqref{condition2}:
\begin{align}\label{condition3}
        &\text{
        Given any $\alpha_m\searrow 0 $ strictly, there is $t'\in \mathbb R$ and }\text{$\delta'>0$ such that $t'+\delta'\alpha_m\in E$ for all $m$.
        }
    \end{align}
We will show that the following Proposition \ref{positive} implies Theorem \ref{thm1}. Once this is established, it suffices to prove Proposition \ref{positive}.
\begin{prop}\label{positive}
   If $B\subseteq \mathbb R$ satisfies Condition \eqref{condition3}, then $B$ is somewhere dense.
\end{prop}

\subsubsection{Proof that Proposition \ref{positive} Implies Theorem \ref{thm1}}
Suppose, towards contradiction, that $E$ is not somewhere dense, i.e. $E$ is nowhere dense by equivalence of \eqref{dense1} and \eqref{dense3} of Lemma \ref{lem1}. Let $B=E\cup (-E)$. Since $E$ is nowhere dense, so is $-E$. By the corollary stated at the end of Lemma \ref{lem1}, $B$ is nowhere dense.

To use Proposition \ref{positive}, we check that $B$ satisfies Condition \eqref{condition3}. Let $\alpha_m\searrow 0$ strictly. Since $E$ satisfies Condition \ref{condition2}, there is $\delta\neq 0$ and $t\in \mathbb R$ such that $t+\delta\alpha_m\in E$ for all $m$. If $\delta>0$, then $t+\delta\alpha_m\in E\subseteq B$; if $\delta<0$, then $-t+(-\delta)\alpha_m\in -E\subseteq B$, so in either case, $B$ satisfies Condition \eqref{condition3}.

By Proposition \ref{positive}, $B$ is somewhere dense, which is a contradiction by equivalence of \eqref{dense1} and \eqref{dense3} of Lemma \ref{lem1} as we showed above that $B$ is nowhere dense.

{\bf Remark: }To avoid excessive use of extra terminology, from now on we will not be referring to Proposition \ref{positive} itself in the subsequent argument. Instead, we will assume without loss of generality that $\delta>0$ in the assumption of Theorem \ref{thm1}.

\section{A Cantor-like Construction}
The main idea of proving Theorem \ref{thm1} is by contradiction. To achieve the contradiction, we will assume that $E$ is nowhere dense, and construct a Cantor-like set containing $E$. At each level of construction of the Cantor set, we are removing intervals with specific lengths from the middle thirds of the remaining intervals. We then construct a slowly decreasing sequence $\{\alpha_m\}$, with rate of decrease depending on the lengths of the removed intervals, such that $E$ contains no affine copy of $\{\alpha_m\}$. This construction will be the key to our proof of Theorem \ref{thm1}.

We will use the following standard notations and definitions:

\subsection{The main construction}One of the main steps in the proof of Theorem \ref{thm1} is the following Cantor-type construction.
\begin{prop}\label{prop1}
Let $A\subseteq [0,1]$ be nowhere dense. Then there is a countable collection of open sets $\{O_n:n\geq 1\}$ and a countable collection of closed intervals $\{K_{n,j}:n\geq 1,1\leq j\leq 2^n\}$, with the following properties:
\begin{enumerate}[label=(\alph*)]
\item \label{cond1} $A\subseteq [0,1]\backslash (\cup_{i=1}^n O_i)$ for each $n\geq 1$.
\item \label{cond2}$\overline{O}_n\cap \overline{O}_{n'}=\varnothing$ for all $n\neq n'$.
\item \label{cond3}Each $O_n$ is of the form
\begin{equation}\label{star0}   O_n=\bigcup_{j=1}^{2^{n-1}}I_{n,j}, 
\end{equation}
where for each $n$, $\{I_{n,j}:1\leq j\leq 2^{n-1}\}$ is a collection of open intervals of the same length (denoted by $l_n$) with disjoint closures. Without loss of generality, $l_n$ can be chosen to be decreasing to $0$ such that $l_n^{-1}\in \mathbb N$.
\item \label{cond4} For each $n$, $[0,1]\backslash \cup_{i=1}^n O_i$ is a disjoint union of $2^n$ closed intervals, which we denote as $\{K_{n,j}:1\leq j\leq 2^n\}$ from left to right. They obey the relation $[0,1]\backslash \cup_{i=1}^n O_i=\cup_{j=1}^{2^n}K_{n,j}$, or equivalently, $[0,1]\backslash \cup_{i=1}^n \overline{O}_i=\cup_{j=1}^{2^n}\mathrm{Int}(K_{n,j})$. In addition, $|K_{n,j}|<(2/3)^n$ for each $n$ and each $1\leq j\leq 2^n$.
\end{enumerate}
As a consequence, 
\begin{equation}\label{Asubset}
    A\subseteq [0,1]\backslash \left(\bigcup_{n=1}^\infty O_n\right)=\bigcap_{n=1}^\infty \bigcup_{j=1}^{2^n}K_{n,j}.
\end{equation}
\end{prop}
\begin{proof}
We construct $O_n$ inductively. In the first step, by \eqref{dense2} of Lemma \ref{lem1} applied to $A$ with $K=I=[0,1]$, we can find an open interval $I_{1,1}\subseteq [1/3,2/3]$ which lies in $A^c$. Let the length of $I_{1,1}$ be $l_1$ (since we can always take a shorter interval within $I_{1,1}$, we may assume $l_1^{-1}\in \mathbb N$), and let $O_1:=I_{1,1}$. Note that $[0,1]\backslash O_1$, which contains $A$, has $2$ closed connected components, which we denote as $K_{1,1}$ and $K_{1,2}$ from left to right (See Figure \ref{fig1}). By construction, $[0,1/3]\subseteq K_{1,1}\subseteq [0,2/3)$, so $1/3\leq |K_{1,1}|<2/3$; similarly we also have $1/3\leq |K_{1,2}|<2/3$. Hence all \ref{cond1}-\ref{cond4} are satisfied for $n=1$ (\ref{cond2} being null here).
\begin{figure}[!h]
\begin{center}
\begin{tikzpicture}
\begin{axis}[
axis lines*=middle,
 axis line style=\empty,
xmin=0, xmax=1,
ymin=-1, ymax=5,
 tick style={color=black},
 xtick={0,0.333,0.37,0.58,0.667,1},
 xticklabels={$0$,$\frac 1 3$ , , ,$\frac 2 3$, $1$},
 axis y line=none,
 ytick=\empty,
 width=15cm,
  height=10cm,
 ]
 \draw[decoration={brace,raise=0.8cm,amplitude=15pt},decorate]
 (0,0) -- node[above=1.5cm] {$K_{1,1}$} (0.37,0);
	
 \addplot[line width=0.1cm, samples=100, domain=0.37:0.58]	({x}, {0});
\node[right] at (axis cs:0.41,-0.5) {$I_{1,1}=O_1$};

 \draw[decoration={brace,raise=0.8cm,amplitude=15pt},decorate]
 (0.58,0) -- node[above=1.5cm] {$K_{1,2}$} (1,0);
 
 \addplot [only marks, mark=o] table {
0.3333 0
0.6667  0
};
\end{axis}
\end{tikzpicture}

\caption{Removing an interval $I_{1,1}$ from the middle third of $[0,1]$.}\label{fig1}
\end{center}
\end{figure}
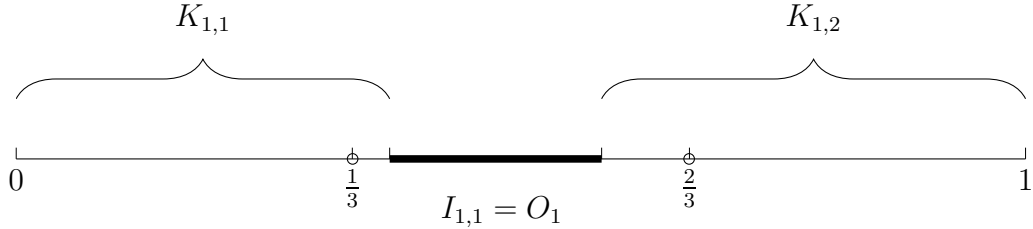

In general, at the end of the $n$-th step, we have obtained $O_n$ and hence $I_{n,j}$ and $K_{n,j}$ obeying the requirements \ref{cond1}-\ref{cond4}. In the $(n+1)$-th step, we apply \eqref{dense2} of Lemma \ref{lem1} to $A$ for each $1\leq j\leq 2^n$ with $I=K_{n,j}$ and find an open sub-interval $I_{n+1,j}$ of the closed middle third of $K_{n,j}$ contained in $A^c$. A priori the intervals $I_{n+1,j}$ may have varying lengths. If $l>0$ with $l^{-1}\in \mathbb N$ and $l\leq \min\{l_n/2,|I_{n+1,1}|,\dots,|I_{n+1,2^n}|\}$, we replace each $I_{n+1,j}$, $1\leq j\leq 2^n$ by a subinterval of length $l$, and we define $l_{n+1}=l$. By a slight abuse of notation we continue to call these smallest subintervals $I_{n+1,j}$. Thus all $I_{n+1,j}$'s now have the same lengths $l_{n+1}\leq l_n/2$, such that $l_{n+1}^{-1}\in \mathbb N$ and that $l_{n}\to 0$. 

(Refer to Figure \ref{fig2}, which demonstrates for a fixed $K_{n,j}$ two subsequent iterations. We remark here that the two solid dots denote the trisection points of $K_{n,j}=[a,b]$. Similarly, the four empty dots denote the trisection points of $K_{n+1,2j-1}$ and $K_{n+1,2j}$, respectively.)

Since for each $1\leq j\leq 2^n$, $\overline{I}_{n+1,j}$ lies in the closed middle third $\tilde K_{n,j}$ of the closed interval $K_{n,j}$, and $\{K_{n,j}:1\leq j\leq 2^n\}$ are disjoint by \ref{cond4} in the $n$-th step, we see that $\{\overline{I}_{n+1,j}:1\leq j\leq 2^n\}$ are disjoint. Furthermore, $\cup_{j=1}^{2^n}\overline{I}_{n+1,j}$ is disjoint from $\cup_{i=1}^n \overline{O}_i$ since by the $n$-th step we have 
$$\bigcup_{i=1}^n \overline{O}_i=[0,1]\backslash\bigcup_{j=1}^{2^n}\mathrm{Int}(K_{n,j})\subseteq [0,1]\backslash\bigcup_{j=1}^{2^n}\tilde K_{n,j}\subseteq [0,1]\backslash \bigcup_{j=1}^{2^n}\overline{I}_{n+1,j}.
$$
Let $O_{n+1}:=\bigcup_{j=1}^{2^n}I_{n+1,j}$ be the disjoint union of these open intervals, and by disjointness we also have $\overline{O}_{n+1}:=\cup_{j=1}^{2^n}\overline{I}_{n+1,j}$. Then we have just showed that
\begin{equation}\label{star2}
    \overline{O}_{n+1}\cap \overline{O}_{i}=\varnothing,
\end{equation}
for all $1\leq i\leq n$.

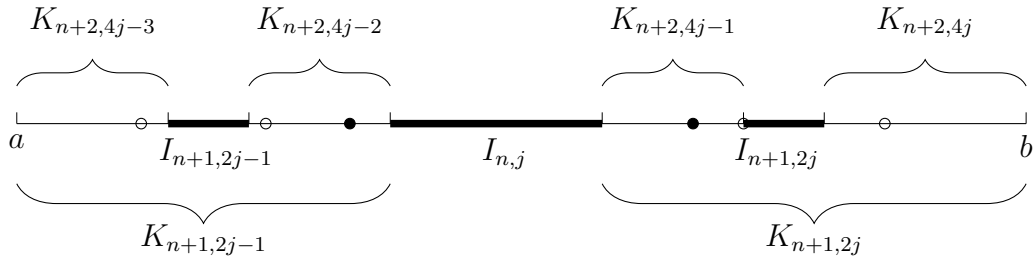
\begin{figure}[!h]
\begin{center}
\begin{tikzpicture}
\begin{axis}[
axis lines*=middle,
 axis line style=\empty,
xmin=0, xmax=1,
ymin=-4, ymax=4,
 tick style={color=black},
 xtick={0,0.15,0.23,0.37,0.58,0.72,0.8, 1},
 xticklabels={$a$, , , , , , ,$b$},
 axis y line=none,
 ytick=\empty,
 width=15cm,
  height=8cm,
 ]
 \draw[decoration={brace,raise=0.5cm,amplitude=10pt},decorate]
 (0,0) -- node[above=1cm] {$K_{n+2,4j-3}$} (0.15,0);
 
 
 \addplot[line width=0.1cm, samples=100, domain=0.15:0.23]	({x}, {0});
\node[right] at (axis cs:0.13,-0.5) {$I_{n+1,2j-1}$};

 \draw[decoration={brace,raise=0.5cm,amplitude=10pt},decorate]
 (0.23,0) -- node[above=1cm] {$K_{n+2,4j-2}$} (0.37,0);

 
\addplot[line width=0.1cm, samples=100, domain=0.37:0.58]	({x}, {0});
\node[right] at (axis cs:0.45,-0.5) {$I_{n,j}$};

\draw[decoration={brace,raise=0.5cm,amplitude=10pt},decorate]
 (0.58,0) -- node[above=1cm] {$K_{n+2,4j-1}$} (0.72,0);
 \addplot[line width=0.1cm, samples=100, domain=0.72:0.8]({x}, {0});
 \node[right] at (axis cs:0.7,-0.5) {$I_{n+1,2j}$};
 
 \draw[decoration={brace,raise=0.5cm,amplitude=10pt},decorate]
 (0.8,0) -- node[above=1cm] {$K_{n+2,4j}$} (1,0);
 \draw[decoration={brace,mirror,raise=0.8cm,amplitude=15pt},decorate]
 (0,0) -- node[below=1.2cm] {$K_{n+1,2j-1}$} (0.37,0);
 \draw[decoration={brace,mirror,raise=0.8cm,amplitude=15pt},decorate]
 (0.58,0) -- node[below=1.2cm] {$K_{n+1,2j}$} (1,0);
 \addplot [only marks, mark=o] table {
0.1233 0
0.2467  0
0.72  0   
0.86  0
};
\addplot [only marks, mark=*] table {
0.33 0
0.67 0
};
\end{axis}
\end{tikzpicture}

\caption{Two further iterations applied to $K_{n,j}=[a,b]$ (trisection points indicated).}\label{fig2}
\end{center}
\end{figure}
We now proceed to verify conditions \ref{cond1}-\ref{cond4}. We start with \ref{cond1}. Since $A\subseteq [0,1]\backslash (\cup_{i=1}^n O_i)$ by induction hypothesis, it suffices to show that
\begin{equation}\label{star1}
    A\subseteq [0,1]\backslash O_{n+1}.
\end{equation}
However, $O_{n+1}$ was chosen as the union of intervals $I_{n+1,j}$, all of which are disjoint from $A$. Hence
\eqref{star1} follows. 

We proceed to \ref{cond2}. In view of the induction hypothesis, this would follows if we show that $\overline{O}_{n+1}\cap \overline{O}_i=\varnothing$ for $i=1,\dots,n$. But this is \eqref{star2} that we have proved.

Part \ref{cond3} follows by definition of $O_{n+1}$ and disjointness of $\{\overline{I}_{n+1,j}:1\leq j\leq 2^n\}$.

For \ref{cond4}, since up to the $n$-th step we have $2^n$ intervals $K_{n,j}$, and given $1\leq j\leq 2^n$, each $K_{n,j}\backslash I_{n,j} $ is a union of $2$ disjoint closed intervals, we see $[0,1]\backslash \cup_{i=1}^{n+1}O_i$ is a disjoint union of $2^{n+1}$ closed intervals, which we denote as $K_{n+1,j},1\leq j\leq 2^{n+1}$ from left to right.

Note that with our choice of indices, we have $K_{n,j}\backslash I_{n,j}=K_{n+1,2j-1}\cup K_{n+1,2j}$. We write $K_{n,j}=[a,b]$, $I_{n,j}=(c,d)$, then $K_{n+1,2j-1}=[a,c]$. Since $I_{n,j}$ is a subinterval of the middle third of $K_{n,j}$, we have 
$$
|K_{n+1,2j-1}|=c-a< \tfrac 2 3(b-a)=\tfrac 2 3|K_{n,j}|.
$$
By the induction hypothesis, we have $|K_{n,j}|<(2/3)^n$, so $|K_{n+1,2j-1}|<(2/3)^{n+1}$. Similarly we can show $|K_{n+1,2j}|<(2/3)|K_{n,j}|<(2/3)^{n+1}$. As this holds for all $1\le j\leq 2^n$, we see that $|K_{n+1,j|}<(2/3)^{n+1}$ for all $1\leq j\leq 2^{n+1}$.

Hence the induction closes. Lastly, letting $n\to \infty$ shows that 
\begin{align*}
  A&\subseteq [0,1]\backslash \left(\bigcup_{n=1}^\infty O_n\right)
  =[0,1] \bigcap \left(\bigcap_{n=1}^\infty O_n^c\right)
  =\bigcap_{n=1}^\infty \left([0,1]\cap O_n^c\right)
  =\bigcap_{n=1}^\infty \bigcup_{j=1}^{2^n}K_{n,j}.  
\end{align*}
\end{proof}

The proof of Proposition \ref{prop1} shows that any interval $K_{n,j}$ from the $n$-th step of the construction yields exactly two intervals $K_{n+1,2j-1}$ and $K_{n+1,2j}$ at the $n$-th step, i.e.
$$
K_{n+1,r}\subseteq K_{n,j}\quad \text{if and only if}\quad r\in \{2j-1.2j\}.
$$
Moreover, if $K_{n,j}=[a,b]$, then $a\in K_{n+1,2j-1}$, $b\in K_{n+1,2j}$.

We will refer to $K_{n+1,2j-1}$ and $K_{n+1,2j}$ as the ``children" of $K_{n,j}$. Each interval $K_{n,j}$ generates exactly $2^k$ descendants after $k$ subsequent steps. The rightmost of these intervals is $K_{n+k,2^k j}$. For fixed $n$ and $j$, as $k$ increases, the closed and bounded intervals $\{K_{n+k,2^k j}:k\geq 1\}$ form a decreasing nested sequence such that each $K_{n+k,2^k j}$, $k\geq 1$ contains the right endpoint of $K_{n,j}$, namely, $\sup K_{n,j}$. Additionally, in view of \ref{cond4}, we have $|K_{n+k,2^k j}|<(2/3)^{n+k}\to 0$. Hence the nested interval property leads to the following lemma:

\begin{lem}\label{Knj}
Fix $n\geq 1$, $1\leq j\leq 2^n$. Then
$$
\sup_{k\geq 1}(\inf K_{n+k,2^k j})=\lim_{k\to \infty}(\inf K_{n+k,2^k j})=\sup K_{n,j}.
$$
\end{lem}

\subsection{Distribution of the deleted open sets}
The following set relation will be used in the last part of the proof of Lemma \ref{emp} which leads to the main theorem. Recall the left neighbourhood $B_-$ and the $I^*$ notation introduced in Section \ref{2not}.
\begin{prop}\label{contain}
   The sets $\{O_n:n\geq 1\}$ constructed in the proof of Proposition \ref{prop1} obey the following property: for $N\geq 1$,
   \begin{equation}\label{union}
\bigcup_{n=N+1}^\infty B_-\left(O_n,\left(\frac 2 3\right)^n\right)\supseteq [0,1)\backslash \left(\bigcup_{n=1}^N O_n^*\right)=\bigcup_{j=1}^{2^N}K_{N,j}^*. 
\end{equation}
In other words, the intervals $\{I_{n,j}\}$ are densely distributed; if some $x$ is not covered by any of the $O_n^*$'s up to stage $N$, then there is some $n\geq N+1$ and some $j$ so that $x$ will be within the left $(2/3)^n$-neighbourhood of $I_{n,j}$.
\end{prop}    

The proof of this proposition is based on the following simple observation.
\begin{lem}\label{left}
Let $K$ be a closed interval, and let $\tilde K$ denote its closed middle third. Then for each open interval $I\subseteq\tilde K$, we have
\begin{equation*}
B_-\left(I,\tfrac 2 3|K|\right)\supseteq [\inf K,\sup I).
\end{equation*}
\end{lem}
(The illustration of this lemma and the proof is shown in Figure \ref{fig3}.)
\begin{proof}
Let $K=[a,b]$ and $I=(c,d)$.  By \ref{not1} of Proposition \ref{Bminus}, we have
$$
B_-\left(I,\tfrac 2 3|K|\right)
=\left(c-\tfrac 2 3|K|,d\right).
$$
Since $I\subseteq \tilde K$, we have $c< a+2(b-a)/3$. Hence
$$
c-\tfrac 2 3|K|<a+\tfrac 2 3 (b-a)-\tfrac 2 3 (b-a)=a.
$$
Thus we have $B_-\left(I,\frac 2 3|K|\right)\supseteq [a,d)=[\inf K,\sup I)$.
\end{proof} 
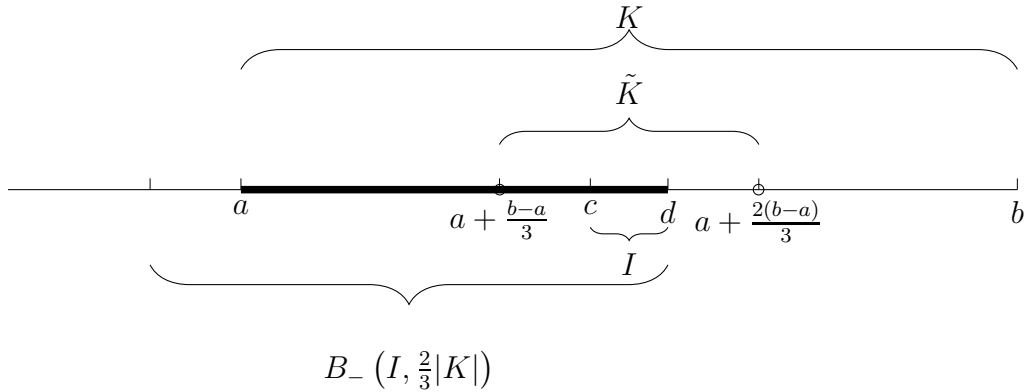
\begin{figure}[!h]
\begin{center}
\begin{tikzpicture}
\begin{axis}[
axis lines*=middle,
 axis line style=\empty,
xmin=-0.3, xmax=1,
ymin=-5, ymax=4,
 tick style={color=black},
 xtick={-0.117,0,0.333,0.45,0.55,0.667,1},
 xticklabels={,$a$, $a+\frac {b-a}3$,$c$,$d$,$a+\frac {2(b-a)}3$,$b$},
 axis y line=none,
 ytick=\empty,
 width=15cm,
  height=8cm,
 ]
 \draw[decoration={brace,raise=1.6cm,amplitude=15pt},decorate]
 (0,0) -- node[above=2cm] {$K$} (1,0);
 \draw[decoration={brace,raise=0.6cm,amplitude=10pt},decorate]
 (0.333,0) -- node[above=1cm] {$\tilde K$} (0.667,0);
 \draw[decoration={brace,mirror,raise=0.5cm,amplitude=5pt},decorate]
 (0.45,0) -- node[below=0.7cm] {$I$} (0.55,0);
 \draw[decoration={brace,mirror, raise=1cm,amplitude=15pt},decorate]
 (-0.117,0) -- node[below=2cm] {$B_-\left(I,\frac 2 3 |K|\right)$} (0.55,0);
 \addplot [only marks, mark=o] table {
0.3333 0
0.6667 0
};

\addplot[line width=0.1cm, samples=100, domain=0:0.55]({x}, {0});
\end{axis}
\end{tikzpicture}

\caption{Illustration of Lemma \ref{left}, with $[a,d)=[\inf K,\sup I)$ shaded}\label{fig3}
\end{center}
\end{figure}

Now we can give a proof of Proposition \ref{contain}.
\begin{proof}
Fix $N$. Recall that \ref{cond4} of Proposition \ref{prop1} gives that for each $N$, $[0,1)\backslash \left(\cup_{n=1}^N O_n\right)=\cup_{j=1}^{2^N}K_{N,j}$. Since $\{K_{N,j}:1\leq j\leq 2^N\}$ are disjoint, using our definition of $I^*$ for each interval $I$ introduced above, we also have $[0,1)\backslash \left(\cup_{n=1}^N O_n^*\right)=\cup_{j=1}^{2^N}K_{N,j}^*$.

Fix $N,j$ and consider a single $K_{N,j}$ (See Figure \ref{fig2} again). For $k\geq 1$, since the middle third of $K_{N+k-1,2^{k-1} j}$ contains $I_{N+k,2^{k-1} j}$, by Lemma \ref{left} applied to $K_{N+k-1,2^{k-1}j}$, we have
\begin{equation}\label{star3}
    B_-\left(I_{N+k,2^{k-1}j},\tfrac 2 3 |K_{N+k-1,2^{k-1}j}|\right)\supseteq [\inf K_{N+k-1,2^{k-1}j},\sup I_{N+k,2^{k-1}j}).
\end{equation}
Again, since $I_{N+k,2^{k-1}j}$ is deleted from $K_{N+k-1,2^{k-1}j}$ whose ``child" on the right is $K_{N+k,2^k j}$, we have
\begin{equation}\label{star4}
\sup I_{N+k,2^{k-1}j}=\inf K_{N+k,2^{k}j}.
\end{equation}

Taking union over $k\geq 1$ on both sides in \eqref{star3}, we have
\begin{align*}
 \bigcup_{k=1}^\infty B_-\left(I_{N+k,2^{k-1} j},\frac 2 3 |K_{N+k-1,2^{k-1} j}|\right)
&\supseteq\bigcup_{k=1}^\infty [\inf K_{N+k-1,2^{k-1}j},\sup I_{N+k,2^{k-1}j})\\
 (\text{by } \eqref{star4})&=\bigcup_{k=1}^\infty [\inf K_{N+k-1,2^{k-1}j},\inf K_{N+k,2^{k}j}).
\end{align*}
We observe that for each $k$, the $k$-th interval above is adjacent to the $(k+1)$-th one. As a result, the union is a single interval given by
\begin{align*}
    [\inf K_{N,j},\sup_{k\geq 1}(\inf K_{N+k,2^{k}j})).
\end{align*}
But by Lemma \ref{Knj}, $\sup_{k\geq 1}(\inf K_{N+k,2^{k}j})=\sup K_{N,j}$, so $[\inf K_{N,j},\sup_{k\geq 1}(\inf K_{N+k,2^{k}j}))=[\inf K_{N,j},\sup K_{N,j})=K_{N,j}^*$. What we have just shown is then
\begin{equation}\label{star5}
    \bigcup_{k=1}^\infty B_-\left(I_{N+k,2^{k-1} j},\frac 2 3 |K_{N+k-1,2^{k-1} j}|\right)\supseteq K^*_{N,j}.
\end{equation}

Thus the left hand side of \eqref{union} is equal to:
\begin{align*}
\bigcup_{n=N+1}^\infty B_-\left(O_n,\left(\frac 2 3\right)^n\right)
&=\bigcup_{k=1}^\infty B_-\left(O_{N+k},\left(\frac 2 3\right)^{N+k}\right)\\
(\text{by \eqref{star0} in \ref{cond3} of Proposition \ref{prop1}})&=\bigcup_{k=1}^\infty B_-\left(\bigcup_{l=1}^{2^{N+k-1}}I_{N+k,l},\left(\frac 2 3\right)^{N+k}\right)\\
 (\text{by \ref{not2} of Proposition \ref{Bminus}})&\supseteq \bigcup_{k=1}^\infty B_-\left(\bigcup_{j=1}^{2^N}I_{N+k,2^{k-1}j},\left(\frac 2 3\right)^{N+k}\right)\\
 (\text{by \ref{not3} of Proposition \ref{Bminus}})&=\bigcup_{j=1}^{2^N}\bigcup_{k=1}^\infty B_-\left(I_{N+k,2^{k-1}j},\left(\frac 2 3\right)^{N+k}\right)\\
 (\text{by \ref{cond4} of Prop. \ref{prop1} and \ref{not4} of Prop. \ref{Bminus}})&\supseteq \bigcup_{j=1}^{2^N}\bigcup_{k=1}^\infty B_-\left(I_{N+k,2^{k-1} j},\frac 2 3 |K_{N+k-1,2^{k-1} j}|\right)\\
 (\text{by \eqref{star5}})&\supseteq \bigcup_{j=1}^{2^N} K_{N,j}^*.
\end{align*}
\end{proof}

\section{Proof of Theorem \ref{thm1}}

We will prove Theorem \ref{thm1} by contradiction. Suppose $E$ is nowhere dense. For $k\in \mathbb Z$, write 
\begin{equation}\label{defEk}
    E_k=E\cap [k,k+1).
\end{equation}
Then for each $k\in \mathbb Z$, $E_k-k\subseteq [0,1]$ is nowhere dense, so we can use Proposition \ref{prop1} with $A=E_k-k\subseteq [0,1]$ to find $O_n^{(k)}\subseteq [k,k+1]$ and $I_{n,j}^{(k)}\subseteq [k,k+1]$ with lengths $l_n^{(k)}$ as specified by \ref{cond3} of Proposition \ref{prop1}.

\subsection{Constructing a slowly decreasing sequence $\{\alpha_m\}$}\label{alpha_m}
With the countable collection of sequences $\{l_n^{(k)}\}_{n=1}^\infty$ indexed by $k$, we are going to pick an extremely slowly decreasing sequence $\alpha_m\searrow 0$ depending on $\{l_n^{(k)}\}$, such that $E$ does not contain any affine copy of $\{\alpha_m\}$.

Note that for each $k$, $\{l_n^{(k)}\}$ is a sequence in $n$ that decreases to $0$, but the rate may vary for different $k$. By the following lemma, we are going to construct a strictly decreasing sequence $\{\mu_n\}$ which decreases more rapidly than $\{l_n^{(k)}\}$ for any $k$.
\begin{lem}\label{fastest}
For each $k\in \mathbb Z$, let $\{l_n^{(k)}\}_{n=1}^\infty$ with $(l_n^{(k)})^{-1}\in \mathbb N$ be strictly decreasing to $0$. Then there is a sequence $\{\mu_n\}$ with $\mu_n^{-1}\in \mathbb N$ which also decreases strictly to $0$, such that for any $k\in \mathbb Z$ and any $n\geq |k|$ we have $\mu_n\leq l_n^{(k)}$.
\end{lem}	 
\begin{proof}
Let $\mu_n=\min\{l_n^{(k)}:|k|\leq n\}$. Then $\mu_n>0$ for all $n$ since $l_n^{(k)}>0$ for all $k$ and $n$. Also, $\mu_n^{-1}\in \mathbb N$.

We prove that $\{\mu_n\}$ is strictly decreasing. Indeed, let $n\geq 2$, then
\begin{align*}
    \mu_n&=\min\{l_n^{(k)}:|k|\leq n\}\\
    &\leq \min\{l_n^{(k)}:|k|\leq n-1\}\\
    &<\min\{l_{n-1}^{(k)}:|k|\leq n-1\}=\mu_{n-1},
    \end{align*}
where the strict inequality follows since for each $k$, $\{l_n^{(k)}\}$ is strictly decreasing with respect to $n$. Lastly, fix $k\geq 1$. By definition, if $n\geq |k|$, then $\mu_n=\min\{l_n^{(k)}:|k|\leq n\}\leq l_n^{(k)}$.
\end{proof}

Now we start to construct $\{\alpha_m\}$. We set $N_0:=0$ and $N_n:=\mu_n^{-1}+N_{n-1}$ for $n\geq 1$, so $N_n\in \mathbb N$ and increases strictly to $\infty$. 

We then define $\{\alpha_m\}_{m=1}^\infty$ as follows:
\begin{equation}\label{defalpha}
    \alpha_m=\frac 1 {n}-\left(\frac 1 {n}-\frac 1 {n+1}\right)\frac{m-N_{n-1}-1}{N_{n}-N_{n-1}},\quad m=N_{n-1}+1,\dots,N_{n}.
\end{equation}
That is, we set 
\begin{equation}\label{alphan}
 \alpha_1=\alpha_{N_0+1}=1, \quad\alpha_{N_1+1}=\frac 1 2,\quad \alpha_{N_2+1}=\frac 1 3,\quad\dots\quad \alpha_{N_n+1}=\frac 1 {n+1},\quad \dots,   
\end{equation}

and the choice of $\alpha_m$ for intermediate values of $m$ is decided by linearly interpolate between the two closest values, namely, $N_{n-1}+1<m<N_n+1$.

Thus
\begin{equation*}
    \alpha_m-\alpha_{m+1}=\frac {1}{n(n+1)(N_n-N_{n-1})},\quad N_{n-1}+1\leq m\leq N_n.
\end{equation*}

Since $N_n-N_{n-1}=\mu_n^{-1}$ is increasing, it follows that $\alpha_m-\alpha_{m+1}$ is decreasing. Since $\alpha_m-\alpha_{m+1}>0$, we see that $\{\alpha_m\}$ is strictly decreasing.

Refer to Figure \ref{fig4}, which shows the sequence in the case $N_1=4$ and $N_2=8$. For example, $\alpha_m$ decreases from $1$ to $1/2$ in $N_1=4$ steps of equal size $1/2*1/4=1/8$. It then decreases from $1/2$ to $1/3$ in $N_2-N_1=4$ steps of equal size $1/6*1/4=1/24$.

We claim that $A$ contains no affine copy of $\{\alpha_m\}$.

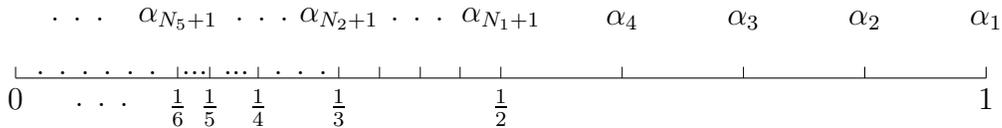
\begin{figure}[!h]\label{fig4}
\begin{center}
\begin{tikzpicture}
\begin{axis}[
axis lines*=middle,
 axis line style=\empty,
xmin=-0.02, xmax=1.02,
ymin=-5, ymax=4,
x axis line style={draw=none},
 tick style={color=black},
 xtick={0,0.167,0.2,0.25,0.333,0.375,0.417,0.458,0.5,0.625,0.75,0.875,1},
 xticklabels={$0$,$\frac 1 6$,$\frac 1 5$, $\frac 1 4$, $\frac 1 3$, , , ,$\frac 1 2$, , , ,$1$},
 axis y line=none,
 ytick=\empty,
 width=15cm,
  height=8cm,
 ]
 \addplot[samples=1000, domain=0:1] 
	({x}, {0});
\addplot [above=0.5cm]
	({1}, {0}) node {$\alpha_1$};
\addplot [above=0.5cm]
	({0.167}, {0}) node {$\alpha_{N_5+1}$};
\addplot [above=0.5cm]
	({0.333}, {0}) node {$\alpha_{N_2+1}$};
\addplot [above=0.5cm]
	({0.5}, {0}) node {$\alpha_{N_1+1}$};
\addplot [above=0.5cm]
	({0.625}, {0}) node {$\alpha_{4}$};
\addplot [above=0.5cm]
	({0.75}, {0}) node {$\alpha_{3}$};
\addplot [above=0.5cm]
	({0.875}, {0}) node {$\alpha_{2}$};
\node[right] at (axis cs:0.05,-0.5) { . . .};
\node[right] at (axis cs:0.025,1.1) { . . .};
\node[right] at (axis cs:0.215,1.1) { . . .};
\node[right] at (axis cs:0.375,1.1) { . . .};
\node[right] at (axis cs:0.01,0.1) {.  .  . . . .};
\node[right] at (axis cs:0.16,0.1) {...};
\node[right] at (axis cs:0.203,0.1) {...};
\node[right] at (axis cs:0.255,0.1) {. . .};
\end{axis}
\end{tikzpicture}
\caption{$\{\alpha_m\}$ when $N_1=4$, $N_2=8$}
\end{center}
\end{figure}

In order to achieve a contradiction, we will prove the following lemma:
\begin{lem}\label{emp}
Let $\{\alpha_m:m\geq 1\}$ be the sequence defined in \eqref{defalpha}. For every $k\geq 1$, $E_k$ denotes the set in \eqref{defEk}. Then for every $\delta>0$ and $m_0\geq 1$, we have
   \begin{equation}\label{empty0}
   [0,1) \bigcap\left(\bigcap_{m=m_0}^\infty (E_k-k)-\delta \alpha_m\right) =\varnothing.
\end{equation}
\end{lem}
The lemma will be proved in Section 5.

\subsection{Proof of Theorem \ref{thm1} assuming Lemma \ref{emp}}

Recall that at the beginning of this section, we have assumed towards contradiction that $E$ is nowhere dense and from this constructed each $E_k$ and a slowly decreasing $\{\alpha_m\}$. To achieve the required contradiction, we will show that $E$ contains no affine copy of $\{\alpha_m\}$.

Suppose, towards contradiction, that there is $t\in \mathbb R$ and $\delta\neq 0$ such that $t+\delta \alpha_m\in E$ for all $m$. Recalling the preliminary reduction in subsection \ref{reduction}, we may assume without loss of generality that $\delta>0$.

Thus there is $k\in \mathbb Z$ such that $E_k$ contains all but finitely many terms of $t+\delta \alpha_m$. Indeed, there is a unique $k\in \mathbb Z$ with $t\in [k,k+1)$. Since $t+\delta \alpha_m\searrow t$, there is $m_0=m_0(\{\alpha_m\},E)$ such that $t+\delta \alpha_m<k+1$ for all $m\geq m_0$, so $t+\delta \alpha_m\in E_k=E\cap [k,k+1]$ for all $m\geq m_0$. Equivalently, $t-k+\delta \alpha_m\in E_k-k\subseteq [0,1]$ for $m\geq m_0$. Letting $m\to \infty$ also shows that $t-k\subseteq [0,1)$. Rewriting this into set notation, we have
$$
t-k\in[0,1) \bigcap\left(\bigcap_{m=m_0}^\infty (E_k-k)-\delta \alpha_m\right),
$$
which is a contradiction to Lemma \ref{emp}. This proves Theorem \ref{thm1}.

\section{Translation of an interval}

In this section, we will prove Lemma \ref{emp}. The main ingredients of this proof are two structural results concerning the union of translation of an interval. These results are contained in Lemma \ref{u1u2} and \ref{ddd} below. The proof of Lemma \ref{emp} assuming these results appear in Section \ref{lastproof}.

Before stating the lemma, we point out a minor simplification of notation. We will temporarily drop the dependence on for every term indexed by $k$ until it becomes necessary. This helps us get rid of using excessively cumbersome notations.

To be more precise, for each $k\geq 1$, let us write $A:=E_k-k\subseteq [0,1]$, and unless otherwise specified, $O_n^{(k)}$, $I_{n,j}^{(k)}$ and $l_n^{(k)}$ (defined at the beginning of this section) will be denoted by $O_n$, $I_{n,j}$ and $l_n$, respectively.

In the new notation, \eqref{empty0} in Lemma \ref{emp} reads
\begin{equation}\label{empty}
    [0,1) \bigcap\left(\bigcap_{m=m_0}^\infty A-\delta \alpha_m\right) =\varnothing.
\end{equation}

\subsection{Structure of union of translates of an interval}\label{structure}
Fix $n$ and we examine carefully $\cup_{m=1}^\infty O_n-\delta\alpha_m$ for a large $n$. Let us recall that $O_n=\cup_{j=1}^{2^{n-1}}I_{n,j}$ from \eqref{star0} of Proposition \ref{prop1}, and fix one connected component $I_{n,j}$ of $O_n$. 

Let
\begin{equation}\label{defM}
 M(n)=M(n,m_0,\delta)=\min\{m\geq m_0:\delta(\alpha_m-\alpha_{m+1})< l_n\}.
\end{equation}
We note that $M(n)$ is finite since $\alpha_m-\alpha_{m+1}\searrow 0$. By the monotonicity of $\alpha_m-\alpha_{m+1}$, for all $m\geq M(n)$, we have $\delta(\alpha_m-\alpha_{m+1})<l_n$. It is worth noting that $M(n)$ depends $\delta$ and $m_0$, but this dependence is suppressed because the subsequent argument does not rely on the specified value of $\delta$ and $m_0$.

\begin{lem}\label{u1u2}
Let $\{\alpha_m\}_{m=1}^\infty$ be a sequence strictly decreasing to $0$ such that $\alpha_m-\alpha_{m+1}$ is also decreasing. Then for any $m_0\geq 1$ and $M(n)$ as in \eqref{defM}, we can decompose the countable union of intervals $\cup_{m=m_0}^\infty I_{n,j}-\delta \alpha_m$ into a disjoint union of $U_1$ and $U_2$, where
\begin{equation*}
    U_1=U_1(j)=\bigcup_{m=m_0}^{M(n)-1}I_{n,j}-\delta \alpha_m
\end{equation*}
    is a disjoint union of open intervals of the same length $l_n$, and
\begin{equation*}
    U_2=U_2(j)=\bigcup_{m=M(n)}^{\infty}I_{n,j}-\delta \alpha_m
\end{equation*}
is a single open interval with length $l_n+\delta \alpha_{M(n)}$ and the same right endpoint as $I_{n,j}$. Using our $B_-$ notation, this can be written as
\begin{equation}\label{u2}
    U_2=B_-(I_{n,j},\delta \alpha_m).
\end{equation}
\end{lem} 
This lemma is illustrated in Figure \ref{fig5}. In this figure, we first fix an interval $I=I_{n,j}$ and show the relative positions of $I-\delta\alpha_m$ for different choices of $m\geq m_0$. To showcase the threshold for the overlapping phenomenon, we draw these intervals indexed by $m$ along the vertical axis.

We also remark that $U_1$ and $U_2$ again depend on $n,j$ (as well as $\delta$ and $m_0$), but we suppress the dependence for the moment since for now we will be only considering one single $I_{n,j}$. Another crucial observation is that our $M(n)$ is independent of the choice of $j$, so it works for all intervals $\{I_{n,j},1\leq j\leq 2^{n-1}\}$ in the $n$-th iteration of the construction in the proof of Proposition \ref{prop1}. In the future, we call $U_1$ the disjoint part and $U_2$ the overlapping part.

\begin{figure}[!h]
\begin{center}
\begin{tikzpicture}
\begin{axis}[
axis lines=middle,
 axis line style={->},
xmin=-22, xmax=2,
ymin=-1.5, ymax=10,
 tick style={color=black},
 xtick=\empty,
 ytick={1,2,3,4,5,6},
 xlabel=$\mathbb R$,
 ylabel=$m-m_0+1$,
 width=15cm,
  height=8cm,
]

\addplot[line width=0.1cm, samples=100, domain=-7*3:-6*3]	({x}, {0});
\addplot[line width=0.1cm, samples=100, domain=-7*3:-6*3]	({x}, {1});
\node[right] at (axis cs:-21.5,-0.75) {$I-\delta \alpha_{m_0}$};
  
\addplot[line width=0.1cm,samples=100, domain=-5*3:-4*3]({x}, {0});
\addplot[line width=0.1cm, samples=100, domain=-5*3:-4*3]({x}, {2});
\node[right] at (axis cs:-15.5,-0.75) {$I-\delta \alpha_{m_0+1}$};

\addplot[line width=0.1cm,samples=100, domain=-3.5*3:-2.5*3]({x}, {0});
\addplot[line width=0.1cm,samples=100, domain=-3.5*3:-2.5*3]({x}, {3});
\node[right] at (axis cs:-11,-0.75) {$I-\delta \alpha_{m_0+2}$};

\addplot[line width=0.1cm,samples=100, domain=-2.5*3:-1.5*3] ({x}, {4});
\addplot[line width=0.1cm, blue,samples=100, domain=-2.5*3:-4.5] ({x}, {0.15});

\addplot[line width=0.1cm,samples=100, domain=-2*3:-1*3] 
({x}, {5});
\addplot[line width=0.1cm,red,samples=100, domain=-2*3:-3] 
({x}, {-0.15});

\addplot[line width=0.15cm,green,samples=100, domain=-2*3:-4.5] 
({x}, {0});

\addplot[line width=0.1cm,samples=100, domain=-1.5*3:-0.5*3] ({x}, {6});
\node[right] at (axis cs:-3,7.5) {.  .  .};
	
\addplot[line width=0.1cm,samples=100, domain=-1.0*3:-0]({x}, {9.7});

\addplot +[mark=none] coordinates {(-7.5, 0) (-7.5, 10)};

\node[right] at (axis cs:-15,7) {$U_1$};
\node[right] at (axis cs:-4.5,7) {$U_2$};

\addplot[line width=0.1cm,samples=100, domain=-1.0*3:-0]({x}, {0}); 
\node[right] at (axis cs:-2,-0.75) {$I$};
	
\addplot[dashed, samples=100, domain=0:5]({-6}, {x}); 	
\addplot[dashed, samples=100, domain=0:4]({-4.5}, {x}); 
\addplot[dashed, samples=100, domain=0:10]({-3}, {x});

\end{axis}
\end{tikzpicture}
\caption{Structure of $\cup_{m=m_0}^\infty I_{n,j}-\delta \alpha_m$ when $M(n)=m_0+3$}\label{fig5}
\end{center}
\end{figure}
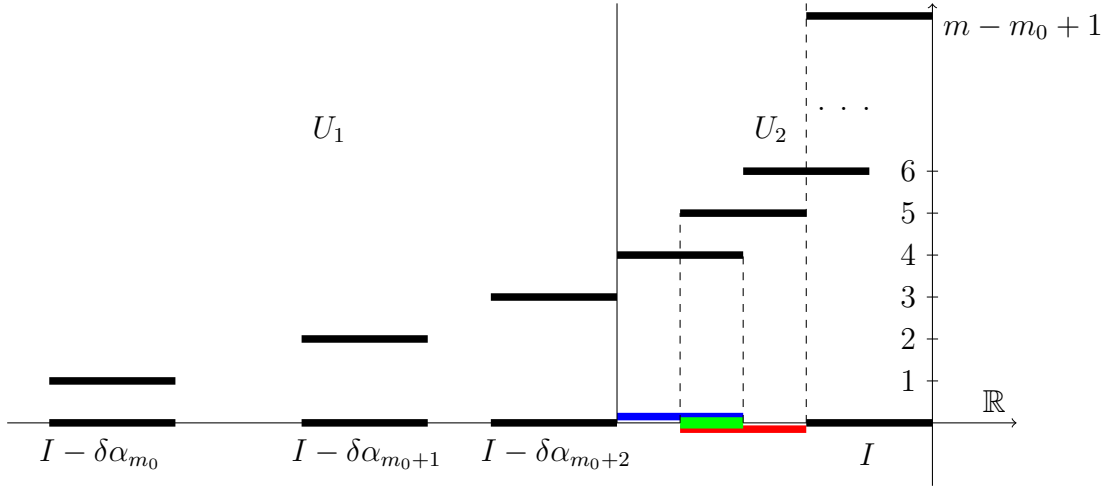

\begin{proof}[Proof of Lemma \ref{u1u2}]
As all $I_{n,j}-\delta \alpha_m$ are open intervals and $\alpha_m$ is strictly decreasing, $U_1$ is a disjoint union if and only if for each $m_0\leq m\leq M(n)-2$, we have $\sup I_{n,j}-\delta \alpha_m\leq \inf I_{n,j}-\delta \alpha_{m+1}$. This is true if and only if $\delta (\alpha_m-\alpha_{m+1})\geq \sup I_{n,j}-\inf I_{n,j}=l_n$ for all $1\leq m\leq M(n)-2$, which follows from the definition \eqref{defM} of $M(n)$. Since $\{I_{n,j}-\delta \alpha_m:1\leq m\leq M(n)-1\}$ are translates of the interval $I_{n,j}$, they have the same length $l_n$.

Since $\delta \alpha_m$ is strictly decreasing, $U_1$ and $U_2$ are disjoint if and only if $I_{n,j}-\delta \alpha_{M(n)-1}$ and $I_{n,j}-\delta \alpha_{M(n)}$ are disjoint. This is true if and only if $\delta (\alpha_{M(n)-1}-\alpha_{M(n)})\geq l_n$, which holds by \eqref{defM}.

The infinite union $U_2$ is a single open interval if and only if for each $m\geq M(n)$, we have $\sup I_{n,j}-\delta \alpha_m> \inf I_{n,j}-\delta \alpha_{m+1}$. This is true if and only if $\delta (\alpha_m-\alpha_{m+1})< \sup I_{n,j}-\inf I_{n,j}=l_n$ for all $m\geq M(n)$, which follows from \eqref{defM}.

Lastly, since $\alpha_m$ decreases strictly to $0$, $\sup I_{n,j}-\delta \alpha_m$ increases strictly to $\sup I_{n,j}$ as $m\to \infty$. Since we have shown that $U_2$ is an open interval, we have $U_2=(\inf I_{n,j}-\delta \alpha_{M(n)},\sup I_{n,j})$. By Part \ref{not1} of Proposition \ref{Bminus}, we have $U_2=B_-(I_{n,j},\delta \alpha_{M(n)})$, which is \eqref{u2}.
\end{proof}

\subsection{Slow Decay of $\{\alpha_m\}$}
In this subsection, we prove the following lemma, which is a result of the slow decay of $\{\alpha_m\}$.
\begin{lem}\label{ddd}
Let $k\geq 1$. Then there is $n_0=n_0(k,\delta,m_0)$ such that
\begin{equation}\label{ddd1}
    \alpha_{M(n)}\geq (n+1)^{-1},\quad \text{for all } n\geq n_0.
\end{equation}
\end{lem}
Recall that $M(n)$ depends implicitly on $k$.

We first prove that there is $n_0=n_0(k,\delta,m_0)$ such that $M(n)\leq N_n$ for all $n\geq n_0$. (Recall $N_n$ was defined in the construction of $\{\alpha_m\}$ at the end of subsection \ref{alpha_m}, and does not depend on $k$.) Indeed, by definition of $M(n)$, this is true if and only if 
\begin{equation}\label{lasteqn}
  \delta (\alpha_{N_n}-\alpha_{N_n+1})<l_n 
\end{equation}
for all large $n$ such that $N_n\geq m_0$. But by construction of the sequence $\alpha_m$, we have 
$$
\delta(\alpha_{N_n}-\alpha_{N_n+1})=\frac{\delta}{n(n+1)(N_n-N_{n-1})}=\frac{\delta}{n(n+1)\mu_n^{-1}},
$$
which will be strictly less than $\mu_n$ if $n>\delta^{-1}$. But by Lemma \ref{fastest}, $\mu_n\leq l_n:=l_n^{(k)}$ for all $n\geq |k|$. Hence \eqref{lasteqn} holds if $n\geq \max\{\delta^{-1},|k|\}$.

Since $N_n\to \infty$, there is $n_1$ such that $N_n\geq m_0$ for all $n\geq n_1$. Hence we may choose $n_0>\max\{\delta^{-1},|k|,n_1\}$ so that $M(n)\leq N_n$ for all $n\geq n_0$. By monotonicity of $\alpha_m$ and recalling \eqref{alphan}, we have
$$
\alpha_{M(n)}\geq \alpha_{N_n}>\alpha_{N_n+1}=(n+1)^{-1},\quad \text{for all }n\geq n_0,
$$
which is \eqref{ddd1}.

\subsection{A corollary of Lemma \ref{u1u2} and Lemma \ref{ddd}}
In this subsection, we prove the following set relation:
\begin{equation}\label{01}
\bigcup_{m=m_0}^\infty \bigcup_{n=1}^\infty O_n-\delta\alpha_m\supseteq [0,1).
\end{equation}
For the proof of \eqref{01}, we will be only interested in the overlapping part. For each $n$ and $j$, we have
\begin{equation}\label{idk}
    \bigcup_{m=m_0}^\infty I_{n,j}-\delta \alpha_m\supseteq U_2(j)\stackrel{\eqref{u2}}{=}B_-(I_{n,j},\delta \alpha_{M(n)}).
\end{equation}
Recall that $M(n)$ is independent of $j$. Thus we can take the union over $1\leq j\leq 2^{n-1}$ on both sides of \eqref{idk} and obtain
\begin{equation}\label{idk2}
    \bigcup_{j=1}^{2^{n-1}}\bigcup_{m=m_0}^\infty I_{n,j}-\delta \alpha_m\supseteq \bigcup_{j=1}^{2^{n-1}}B_-(I_{n,j},\delta \alpha_{M(n)}).
\end{equation}
Swapping the unions on the left hand side of \eqref{idk2} and by \eqref{star0} and \eqref{cupt}, we see it is equal to $\cup_{m=m_0}^\infty O_n-\delta \alpha_m$. By \eqref{star0} and \ref{not3} of Proposition \ref{Bminus}, the right hand side of \eqref{idk2} is equal to $B_-(O_n,\delta \alpha_{M(n)})$. We have thus showed
\begin{equation}\label{whatever}
    \bigcup_{m=m_0}^\infty O_n-\delta \alpha_m\supseteq B_-(O_n,\delta \alpha_{M(n)}).
\end{equation}
Now we invoke Lemma \ref{ddd} to find an $n_0$ such that $\alpha_{M(n)}\geq (n+1)^{-1}$ for all $n\geq n_0$. We then choose an integer $N\geq n_0$ such that for all $n\geq N$, we have $\delta/(n+1)\geq (2/3)^n$. This implies
\begin{equation}\label{lasteq}
    \delta \alpha_{M(n)}\geq (2/3)^n,\text{ for all $n\geq N$}.
\end{equation}
Taking union over $n$ on both sides of \eqref{whatever}, we have
\begin{align*}
\bigcup_{n=1}^\infty \bigcup_{m=m_0}^\infty O_n-\delta \alpha_m
&\supseteq \bigcup_{n=1}^\infty B_-(O_n,\delta \alpha_{M(n)})\\
&=\left(\bigcup_{n=1}^N B_-(O_n,\delta \alpha_{M(n)})\right)\bigcup \left(\bigcup_{n=N+1}^\infty B_-(O_n,\delta \alpha_{M(n)}) \right)\\
(\text{by \ref{not4} of Prop. \ref{Bminus} and \eqref{lasteq}})&\supseteq\left(\bigcup_{n=1}^N B_-(O_n,\delta \alpha_{M(n)})\right)\bigcup \left(\bigcup_{n=N+1}^\infty B_-\left(O_n,\left(\frac 2 3\right)^n\right)\right)\\
(\text{by \eqref{union} in Proposition \ref{contain}})&\supseteq \left(\bigcup_{n=1}^N B_-(O_n,\delta \alpha_{M(n)})\right)\bigcup \left([0,1)\backslash \left(\bigcup_{n=1}^N O_n^*\right)\right) \\
(\text{by \ref{not1a} of Proposition \ref{Bminus}})&\supseteq \left(\bigcup_{n=1}^N O_n^*\right)\bigcup \left([0,1)\backslash \left(\bigcup_{n=1}^N O_n^*\right)\right)\supseteq [0,1).
\end{align*}
Hence \eqref{01} follows.

\subsection{Proof of Lemma \ref{emp}}\label{lastproof}
We can now prove Lemma \ref{emp}, which is expressed in the form \eqref{empty}. By the inclusion relation \eqref{Asubset} in Proposition \ref{prop1}, for any $\delta>0$,
\begin{align*}
    \bigcap_{m=m_0}^\infty A-\delta \alpha_m
    &=\bigcap_{m=m_0}^\infty \left([0,1]\backslash\left( \bigcup_{n=1}^\infty O_n\right)-\delta \alpha_m\right)\\
    &=\bigcap_{m=m_0}^\infty \left([0,1]\bigcap \left(\bigcap_{n=1}^\infty O_n^c\right)-\delta \alpha_m\right)\\
    &\subseteq \bigcap_{m=m_0}^\infty\left(\bigcap_{n=1}^\infty O_n^c-\delta \alpha_m\right)\\
    &=\bigcap_{m=m_0}^\infty\bigcap_{n=1}^\infty (O_n^c-\delta \alpha_m),
\end{align*}
where the last line follows from \eqref{capt}.

Now we take complements in $[0,1)$ on both sides of \eqref{01} showed in the previous section. This gives
\begin{align*}
 \varnothing
 &\supseteq [0,1)\bigcap\left(\bigcap_{m=m_0}^\infty \bigcap_{n=1}^\infty (O_n-\delta\alpha_m)^c\right)\\
 (\text{by \eqref{ct}})&=[0,1)\bigcap\left(\bigcap_{m=m_0}^\infty \bigcap_{n=1}^\infty O_n^c-\delta\alpha_m\right)\\
 (\text{by \eqref{Asubset}})&\supseteq[0,1)\bigcap \left(\bigcap_{m=m_0}^\infty A-\delta \alpha_m\right),
\end{align*}
which is \eqref{empty}. This finishes the proof of Lemma \ref{emp} and thus Theorem \ref{thm1}.

\section{Proof of Theorem \ref{thm2}}
We start with a brief sketch of the proof. First, we introduce the definition of threshold sequences, and then prove Proposition \ref{threshold} which is just Theorem \ref{thm2} with an additional assumption that the prescribed $\{\beta_m\}$ can be replaced by a threshold sequence $\{\eta_m\}$. After that, we will show Proposition \ref{threshold} and  Lemma \ref{betaeta} to be stated below together imply Theorem \ref{thm2}. Lastly we give a proof of Lemma \ref{betaeta}.

\subsection{Threshold sequences}
\begin{defn}[Threshold Sequence]\label{thresseq}
    Let $\{\eta_m\}_{m=1}^\infty$ be a sequence of real numbers. We say $\{\eta_m\}$ is a threshold sequence if it satisfies the following properties:
    \begin{enumerate}
        \item \label{sdec}$\eta_m$ is strictly decreasing.
        \item \label{to0} $\eta_m$ converges to $0$.
        \item \label{decreasing} $\eta_m-\eta_{m+1}\geq\eta_{m+1}-\eta_{m+2},\quad \text{for all }m\geq 1$.
    \end{enumerate}
\end{defn}
    
	\begin{prop}\label{threshold}
	  Let $\{\eta_m\}_{m=1}^\infty$ be a threshold sequence. Then there is a closed and nowhere dense set $A\subseteq [0,1]$, depending on $\{\eta_m\}$, such that for any sequence $\alpha_m\to 0$ with $|\alpha_m|=O(\eta_m)$, there is $\delta>0$ and $t\in \mathbb R$ such that $t+\delta \alpha_m\in A$ for all $m$.
	\end{prop}
	For the demonstration to be more clear, we give a proof of Proposition \ref{threshold} in the next subsection.
	\begin{lem}\label{betaeta}
	Let $\{\beta_m\}$ be a sequence of real numbers strictly decreasing to $0$. Then there is a threshold sequence $\{\eta_m\}$ such that $\beta_m\leq \eta_m$ for all $m$.
	\end{lem}
	
	\subsubsection{Proof that Proposition \ref{threshold} and Lemma \ref{betaeta} imply Theorem \ref{thm2}}

	Let $\{\beta_m\}$ be given as in Theorem \ref{thm2}. By Lemma \ref{betaeta}, find a threshold sequence $\{\eta_m\}$ such that $\beta_m=O(\eta_m)$. By Proposition \ref{threshold} applied to $\{\eta_m\}$, we can find a closed and nowhere dense $A\subseteq [0,1]$, depending on $\{\eta_m\}$, such that for all $|\alpha_m|=O(\eta_m)$, in particular for all $|\alpha_m|=O(\beta_m)=O(\eta_m)$, there is $\delta>0$ and $t\in\mathbb R$ such that $t+\delta \alpha_m\in A$ for all $m$. But  by Lemma \ref{betaeta}, $\{\eta_m\}$ depends on $\{\beta_m\}$ only, so in turn $A$ also depends on $\{\beta_m\}$ only.

	\par
	
	\subsection{Proof of Proposition \ref{threshold}}
	\subsubsection{Construction of $A$}
	We start with any countable collection of open intervals $V_n$ that forms a countable base for the standard topology on $(0,1)$. For example, we can choose $\{V_n\}$ to be the countable collection of all open intervals in $(0,1)$ with rational centres and rational radii. Our set $A$ will be of the form
	\begin{equation}\label{defA}	    	A=[0,1]\backslash\bigcup_{n=1}^\infty J_n
	\end{equation}
	for a carefully chosen collection of intervals $J_n\subseteq V_n$ whose lengths $\lambda_n$ are to be specified (See \eqref{ln}). With this definition, $A\subseteq [0,1]$ is automatically closed and nowhere dense.

	\subsubsection{A measure-theoretic argument}
	We will figure out what conditions can be imposed on $\lambda_n$ so that   the set $A$ we defined satisfies the affine containment property as stated in Proposition \ref{threshold}.
	
	Let $|\alpha_m|=O(\eta_m)$. Assuming $\lambda_n$ has been chosen, we are going to find $\delta>0$ and $t\in \mathbb R$ such that $t+\delta\alpha_m\in A$ for all $m$. In contrast to \eqref{empty}, we show that there is $0<\delta<1$ such that the following set relation holds:
	\begin{equation}\label{nonempty}
   \bigcap_{m=1}^\infty A-\delta\alpha_m \neq\varnothing.
\end{equation}
Using measure theory, \eqref{nonempty} is true if, in particular,
\begin{equation}\label{nonempty2}
    \mathcal L^1\left(\bigcap_{m=1}^\infty A-\delta\alpha_m\right)>0.
\end{equation}
Here, $\mathcal L^1$ denotes the standard Lebesgue measure on $\mathbb R$.

But since $A=[0,1]\backslash (\cup_{n=1}^\infty J_n)$ \eqref{defA}, using \eqref{ct} and \eqref{capt}, we can compute
$$
\bigcap_{m=1}^\infty A-\delta\alpha_m=[0,1]\bigcap \left(\bigcap_{m=1}^\infty\bigcap_{n=1}^\infty J_n^c-\delta\alpha_m\right).
$$
Thus \eqref{nonempty2} holds if and only if (where \eqref{ct} and \eqref{cupt} are used)
\begin{equation*}
1>\mathcal L^1\left([0,1]\backslash\bigcap_{m=1}^\infty\bigcap_{n=1}^\infty J_n^c-\delta\alpha_m\right)=\mathcal L^1\left([0,1] \bigcap \bigcup_{m=1}^\infty\bigcup_{n=1}^\infty J_n-\delta\alpha_m\right). 
\end{equation*}
Hence it suffices to show that there is $\delta>0$ such that
	\begin{equation*}
		1>\mathcal L^1 \left(\bigcup_{m=1}^\infty \bigcup_{n=1}^\infty J_n-\delta\alpha_m\right)=\mathcal L^1 \left(\bigcup_{n=1}^\infty \left(\bigcup_{m=1}^\infty J_n-\delta\alpha_m\right)\right).
		\end{equation*}
It further suffices to show there is $\delta>0$ such that
		\begin{equation}\label{meas}		    \sum_{n=1}^\infty \mathcal L^1\left(\bigcup_{m=1}^\infty J_n-\delta\alpha_m\right)<1.
		\end{equation}
		The following proposition will imply \eqref{meas}:
		\begin{prop}\label{123}
		   \begin{enumerate}
		\item \label{length}
		For any $\delta>0$ and any $n\geq 1$,
	    \begin{equation*}
	       \lim_{\delta\to 0^+} \mathcal L^1\left(\bigcup_{m=1}^\infty J_n- \delta\alpha_m\right)=\lambda_n.
	    \end{equation*}
	    \item \label{monotone}Let $\delta_0>0$ be a fixed constant such that $|\alpha_m|\leq \frac {\eta_m}{2\delta_0}$ for all $m\geq 1$. (Such $\delta_0$ exists since $|\alpha_m|=O(\eta_m)$, and note that $\delta_0$ does not depend on $m,n$.) Then for any $0<\delta<\delta_0$ and any $n\geq 1$,
	    \begin{equation}\label{mono}
	         \mathcal L^1\left(\bigcup_{m=1}^\infty J_n- \delta\alpha_m\right)\leq \mathcal L^1\left(\bigcup_{m=1}^\infty J_n- \eta_m\right).
	    \end{equation}
	    \item \label{bound}
		$$
		\sum_{n=1}^\infty \mathcal L^1\left(\bigcup_{m=1}^\infty J_n- \eta_m\right)<\infty.
		$$
		\end{enumerate}
		\end{prop}
		
		Indeed, if all of the above are true, then by the dominated convergence theorem applied to  $f_\delta(n)=\mathcal L^1(\cup_{m=1}^\infty J_n-\delta\alpha_m)$ with the measure space being the counting measure on $\mathbb N$, we get
		$$
		\lim_{\delta\to 0^+}\sum_{n=1}^\infty \mathcal L^1\left(\bigcup_{m=1}^\infty J_n-\delta\alpha_m\right)=\sum_{n=1}^\infty \lambda_n.
		$$
		Thus \eqref{meas} holds since $\sum_{n=1}^\infty \lambda_n<1$ by \eqref{ln}.
		
\subsubsection{Proof of Proposition \ref{123}}
		We first prove \eqref{length}. Let $\delta>0$ and $n\geq 1$. Denote $J_n:=(a,b)$. Since $\alpha_m\to 0$, it is bounded. Let $c=\inf \{\alpha_m:m\geq 1\}$ and $d=\sup \{\alpha_m:m\geq 1\}$. Then we have $\inf (
		J_n-\delta\alpha_m)=a-\delta\alpha_m\geq a-\delta d$, and $\sup ( J_n-\delta\alpha_m)=b- \delta\alpha_m\leq b-\delta c$. Hence $\cup_{m=1}^\infty J_n- \delta\alpha_m\subseteq (a-\delta d,b-\delta c)$, so 
		$$
		\mathcal L^1\left(\cup_{m=1}^\infty J_n- \delta\alpha_m\right)\leq b-a+\delta(d-c)=\lambda_n+\delta(d-c).$$ 
		On the other hand, $\cup_{m=1}^\infty J_n-\delta\alpha_m\supseteq J_n-\delta\alpha_1=(a-\delta\alpha_1,b-\delta\alpha_1)$, so $\mathcal L^1\left(\cup_{m=1}^\infty J_n- \delta\alpha_m\right)\geq b-a=\lambda_n$. Hence the squeeze law implies that $\mathcal L^1\left(\cup_{m=1}^\infty J_n- \delta\alpha_m\right)$ converges to $\lambda_n$ as $\delta\to 0^+$.

		Now we come to Part \eqref{monotone}. 
		Define, similar to \eqref{defM},
	\begin{equation}\label{newdefM}
	    T(n):=\min\{m:\eta_m-\eta_{m+1}<\lambda_n\}.
	\end{equation}
	Since $\eta_m$ is a threshold sequence (see Definition \ref{thresseq}), it decreases strictly to $0$ and $\eta_{m}-\eta_{m+1}$ is also decreasing. Thus we have $\eta_m-\eta_{m+1}<\lambda_n$ if and only if $m\geq T(n)$.
	
	 By Lemma \ref{u1u2}, we have that $U_1:=\cup_{m=1}^{T(n)-1}J_n-\eta_m$ is a disjoint union of open intervals of length $\lambda_n$, that $U_2:=\cup_{m=T(n)}^\infty 
	 J_n-\eta_m$ is a single open interval of length $\eta_{T(n)}+\lambda_n$, and that $\cup_{m=1}^{T(n)-1}J_n-\eta_m$ and $\cup_{m=T(n)}^\infty J_n-\eta_m$ are disjoint. Thus the right hand side of \eqref{mono} can be computed as:
		\begin{equation}\label{measure}
		    \mathcal L^1\left(\bigcup_{m=1}^\infty J_n- \eta_m\right)=(T(n)-1)\lambda_n+\eta_{T(n)}+\lambda_n=T(n)\lambda_n+\eta_{T(n)}.
		\end{equation}
		
		Now we come to the left hand side of \eqref{mono}. Regardless of the positions of the intervals $\{J_n-\delta\alpha_m\}_{m=1}^{T(n)-1}$, we always have
		$$
		\mathcal L^1\left(\bigcup_{m=1}^{T(n)-1} J_n- \delta\alpha_m\right)\leq \sum_{m=1}^{T(n)-1} \mathcal L^1\left(J_n- \delta\alpha_m\right)= (T(n)-1)\lambda_n.
		$$
		On the other hand, by \ref{monotone} of Proposition \ref{123}, for all $0<\delta<\delta_0$ and for all $m\geq 1$, we have $\delta |\alpha_{m}|\leq \frac{\eta_{m}}2$. Denote $J_n=(a,b)$. Then for all $m\geq T(n)$, we have 
		$$\sup (J_n-\delta\alpha_m)=b-\delta\alpha_m\leq b+\frac {\eta_m}2\leq b+\frac {\eta_{T(n)}}2.$$
		
		Similarly, for all $m\geq T(n)$, we have $\inf (J_n-\delta\alpha_m)\geq a-\frac {\eta_{T(n)}}2$. This implies	$\cup_{m=T(n)}^\infty J_n-\delta\alpha_m\subseteq \left(a-\frac{\eta_{T(n)}}2,b+\frac{\eta_{T(n)}}2\right)$, and so
		$$
		\mathcal L^1\left(\bigcup_{m=T(n)}^{\infty} J_n- \delta\alpha_m\right)\leq \eta_{T(n)}+b-a=\eta_{T(n)}+\lambda_n.
		$$
		Thus
		\begin{align*}
		    \mathcal L^1\left(\bigcup_{m=1}^\infty J_n- \delta\alpha_m\right)
		    &\leq \mathcal L^1\left(\bigcup_{m=1}^{T(n)-1} J_n- \delta\alpha_m\right)+\mathcal L^1\left(\bigcup_{m=T(n)}^{\infty} J_n- \delta\alpha_m\right)\\
		    &\leq (T(n)-1)\lambda_n+\eta_{T(n)}+\lambda_n\\
		    (\text{by \eqref{measure}})&=\mathcal L^1\left(\bigcup_{m=1}^\infty J_n- \eta_m\right).
		\end{align*}
	This finishes the proof of Part \eqref{monotone} of the proposition.
	
	It remains to prove Part \eqref{bound}. By \eqref{measure} this is equivalent to
		\begin{equation}\label{last}
		    \sum_{n=1}^\infty T(n)\lambda_n+\eta_{T(n)}<\infty.
		\end{equation}
	To this end, we need to specify our choice of $\lambda_n$.
	
	Define $K(n):=2\min\{m:\eta_m<n^{-2}\}$. $K(n)$ is well defined since $\eta_m\searrow 0$, and in particular, we have
	\begin{equation}\label{Kn}
	    K(n)\text{ is even }\quad\quad \text{ and }\quad\quad\eta_{\frac{K_n}2}<n^{-2}.
	\end{equation}	
	Recall that $V_n$'s are open intervals that form a topological base for $(0,1)$ and that $J_n$ are chosen to be subintervals of $V_n$ for each $n$.
	
	Then we define: 
	\begin{equation}\label{ln}
	    \lambda_n=\min\left\{|V_n|,2^{-n},\eta_{K(n)}-\eta_{K(n)+1}\right\}>0.
	\end{equation}

	Note that $\lambda_n\leq\eta_{K(n)}-\eta_{K(n)+1}$, so $T(n)>K(n)$ by definition of $T(n)$ in \eqref{newdefM}. By monotonicity of $\{\eta_m\}$ and \eqref{Kn}, we have
	\begin{equation}\label{summable}
	    \sum_{n=1}^\infty \eta_{\left\lfloor \frac{T(n)}2\right\rfloor }\leq \sum_{n=1}^\infty \eta_{\left\lfloor \frac{K(n)}2\right\rfloor }=\sum_{n=1}^\infty \eta_{ \frac{K(n)}2 }<\sum_{n=1}^\infty n^{-2}<\infty.
	\end{equation}
    
	 Also note that since $\eta_m$ is decreasing, $\eta_{T(n)}\leq \eta_{\lfloor T(n)/2\rfloor }$ is also summable by \eqref{summable}.

	The definition of $T(n)$ \eqref{newdefM} implies that for all $m<T(n)$ we have $\eta_m-\eta_{m+1}\geq \lambda_n$. Hence we can bound $T(n)\lambda_n$ from above by:
		\begin{align*}
		    T(n)\lambda_n
		    &=2\frac {T(n)}2 \lambda_n\leq 2\left(T(n)-\left\lfloor \frac {T(n)}2\right \rfloor\right)\lambda_n\\
		    &\leq 2\left(\eta_{\left\lfloor \frac {T(n)}2\right \rfloor}-\eta_{\left\lfloor \frac {T(n)}2\right \rfloor+1}+\cdots+\eta_{T(n)-1}-\eta_{T(n)}\right)\\	    &=2\eta_{\left\lfloor \frac {T(n)}2\right \rfloor}-2\eta_{T(n)},
		\end{align*}
		which is summable by \eqref{summable} and the note following it. This proves \eqref{last}, thus \eqref{bound} of Proposition \ref{123}.

    \subsection{Proof of Lemma \ref{betaeta}}

    Let $\beta_m\searrow 0$ be given. Let $\eta_1=\beta_1$ and $\eta_2=\beta_2$. For $m\geq 3$, we define 
    \begin{equation*}
        \eta_m=\max\left\{\beta_m,2\eta_{m-1}-\eta_{m-2}\right\}.
    \end{equation*}
    By this definition, we have $\eta_m\geq \beta_m$ for all $m\geq 1$ as well as $\eta_{m-1}-\eta_{m}\leq \eta_{m-2}-\eta_{m-1}$ for all $m\geq 3$, which is Part \eqref{decreasing} of Definition \ref{thresseq}. It remains to show Parts \eqref{sdec} and \eqref{to0}, namely, $\eta_m$ strictly decreases to $0$.
    
    We first show by induction that $\eta_m$ is strictly decreasing. First, $\eta_2=\beta_2<\beta_1=\eta_1$. Assuming $\eta_{m-1}<\eta_{m-2}$ for all $m\geq m_0$ where $m_0\geq 3$, we will show that $\eta_{m}<\eta_{m-1}$. We have 2 cases:
    \begin{itemize}
        \item If $\beta_m=\max\left\{\beta_m,2\eta_{m-1}-\eta_{m-2}\right\}$, then $\eta_m=\beta_m<\beta_{m-1}\leq \eta_{m-1}$ as $\beta_m$ is assumed to be strictly decreasing.
        \item  If $2\eta_{m-1}-\eta_{m-2}=\max\left\{\beta_m,2\eta_{m-1}-\eta_{m-2}\right\}$, then $\eta_m=2\eta_{m-1}-\eta_{m-2}<\eta_{m-1}$, since the last inequality equivalent to $\eta_{m-1}<\eta_{m-2}$ which is our induction assumption.
    \end{itemize} 
    Next we show that $\eta_m$ converges to $0$. We have two cases:
    \begin{itemize}
    \item If there is $N\ge 3$ such that for all $m\geq N$, $\beta_m\leq 2\eta_{m-1}-\eta_{m-2}$, then $\eta_m=2\eta_{m-1}-\eta_{m-2}$ for all $m\geq N$. Thus $\{\eta_m:m\geq N-2\}$ is an infinite arithmetic progression of common difference $\eta_{N-1}-\eta_{N-2}<0$ marching to the left. Hence if $m\geq N-2+\frac {\eta_{N-2}}{\eta_{N-2}-\eta_{N-1}}$, then $\eta_m\leq 0$, which is a contradiction since by definition, $\eta_m\geq \beta_m>0$ for all $m$.
    
    \item Otherwise, $\beta_m> 2\eta_{m-1}-\eta_{m-2}$ infinitely often, so there is a subsequence $\eta_{m_k}=\beta_{m_k}$ for all $k$. Since $\beta_m\to 0$, we have $\eta_{m_k}\to 0$. But $\{\eta_m\}$ is a strictly decreasing sequence, so $\{\eta_m\}$ itself also converges to $0$.
        
     \end{itemize}

 \newpage

\section{Appendix}
In the appendix, we give a proof of a particular case of Molter and Yavicoli's result \cite{Molter}. It is actually almost parallel to their proof, but with the notations greatly simplified since we are only considering a special case.
\begin{defn}
		A dimension function $h:[0,\infty)\to [0,\infty]$ is a right-continuous increasing function such that $h(0)=0$, $h(t)>0$ for $t>0$.
	\end{defn}
	\begin{defn}
	    Let $h$ be a dimension function. For a set $E\subseteq \mathbb R$ and $0<\delta\leq \infty$, we define
	    $$
	    \mathcal H^h_\delta(E)=\inf\left\{\sum_i h\left(\mathrm{diam}{B_i}\right) :\bigcup_i B_i\supseteq E,\mathrm{diam}(B_i)<\delta\right\}.
	    $$
	    We then define
	    $$
	    \mathcal H^h(E):=\sup_{0<\delta\leq \infty}\mathcal H^h_\delta(E)=\lim_{\delta\to 0^+}\mathcal H^h_\delta(E).
	    $$
	\end{defn}
	\begin{prop}
		Let $h(x):=-\frac {1}{\ln x}$ with $h(0)=0$ be a dimension function. Suppose $\mathcal H^h(E)=0$ for some set $E\subseteq \mathbb R$. Then $E$ has Hausdorff dimension $0$.
	\end{prop}
	
	\begin{proof}
		Let $s>0$. Then there is $C_s>0$ with $ x^s<-\frac {C_s}{\ln x}$ for all $0\leq x\leq 1$, since $\lim_{x\to 0^+}x^s\ln x=0$ by L'H\^opital's rule. Then for any $0<\delta<1$, 
		\begin{align*}
		\mathcal H^s_\delta(E)&=\inf\left\{\sum_i \left(\mathrm{diam}(B_i)\right)^s :\bigcup_i B_i\supseteq E,\mathrm{diam}(B_i)<\delta\right\}\\
		&\leq \inf\left\{\sum_i C_s h(\mathrm{diam} (B_i)) :\bigcup_i B_i\supseteq E,\mathrm{diam}(B_i)<\delta\right\}\\
		&\leq C_s \mathcal H^h(E)=0.
		\end{align*}
	\end{proof}
    
	\begin{thm}[Theorem 3.2 and Theorem 4.4 of \cite{Molter}, simplified]\label{primitive}
		Let $h$ be any dimension function. 
		Then there is an $F_\sigma$-set $E\subseteq \mathbb R$ such that $\mathcal H^h(E)=0$  and for any sequence $\{\alpha_m\}_{m=1}^\infty\subseteq \mathbb R$, we have
		$$
		\bigcap_{m=1}^\infty E+\alpha_m\neq \varnothing.
		$$
		In particular, $\mathrm{dim}_H(E)=0$ by the previous proposition.
	\end{thm}
	
	\begin{proof}
		Let $M_n\in 2\mathbb N$ be an increasing sequence, $M_1\geq 4$, such that for all $n\geq 2$,
		$$
		h\left(\frac {1}{M_1M_2\cdots M_n}\right)\leq \frac 1 {M_1M_2\cdots M_{n-1}}.
		$$
		
		For each real number $x$, we consider its digit expansion with respect to the sequence $\{M_n\}$:
		$$
		x=[x]+\sum_{n=1}^\infty \frac  {x^{(n)}}{M_1M_2\cdots M_n},\quad 0\leq x^{(n)}\leq M_n-1.
		$$
		where $[x]$ denotes the integral part of $x$. 
		
		Let $F_n$, $n\in \mathbb N$ denote the collection of all real numbers such that its $n$-th digit, $x^{(n)}$, is $0$ or $M_n/2$. If there are two possible expansions of $x$ with one of them having $x^{(n)}= 0$ or $M_n/2$, include that number $x$ in $F_n$ as well (this ensures that $F_n$ is made up of disjoint closed intervals). Let $I_j:=\{(2k-1)2^{j-1}:k\in \mathbb N\}$ for $j\in \mathbb N$. Then $\{I_j\}_{j=1}^\infty$ forms a partition for $\mathbb N$. Define $K_j:=\cap_{n\in I_j}F_n$. For example, $K_2$ is the set of all real numbers so that their $2,6,10,14,\dots$-th digits are $0$ or $M_2/2$. Note $K_j$ is also closed.
		
		Lastly, define $E:=\cup_j K_j$. We claim that $E$ is the required $F_\sigma$-set.
		\begin{itemize}
	    \item
		To show $\mathcal H^h(E)=0$, it suffices to show $\mathcal H^h(K_j)=0$ for all $j$. Let $\delta>0$ be small, and cover $E_j$ by
		$$
		\frac {\prod_{i=1}^{(2k-1)2^{j-1}}M_i} {\prod_{l=1}^k \frac 1 2M_{(2l-1)2^{j-1}}}
		$$ 
		intervals of lengths
		$$
		\frac {1}{\prod_{i=1}^{(2k-1)2^{j-1}}M_i}<\delta.
		$$
		for all large $k$'s.
		Then we have
		\begin{align*}
		\mathcal H^h_\delta (K_j)
		&\leq \frac {\prod_{i=1}^{(2k-1)2^{j-1}}M_i} {\prod_{l=1}^{k} \frac 1 2 M_{(2l-1)2^{j-1}}}\,\,\cdot h\left(\frac {1}{\prod_{i=1}^{(2k-1)2^{j-1}}M_i}\right)\\
		&\leq \frac {\prod_{i=1}^{(2k-1)2^{j-1}}M_i} {\prod_{l=1}^{k} \frac 1 2 M_{(2l-1)2^{j-1}}}\,\,\cdot \frac {1}{\prod_{i=1}^{(2k-1)2^{j-1}-1}M_i}\\
		&=\frac 2{\prod_{l=1}^{k-1} \frac 1 2 M_{(2l-1)2^{j-1}}}\\
		&\leq \frac 1 {2^{k-2}}, \text{for all large }k.
		\end{align*}
		Letting $k\to \infty$, we have $\mathcal H^h_\delta (K_j)=0$. Letting $\delta\to 0^+$, we have then $\mathcal H^h(K_j)=0$. Thus $\mathcal H^h(E)=0$.
		
		\item 
		Now let $\{\alpha_m\}$ be given. We show
		\begin{equation*}
		  \bigcap_{m=1}^\infty E+\alpha_m\neq \varnothing.  
		\end{equation*}
	    We have $E\supseteq K_m$ for all $m\geq 1$, so it suffices to show
	    \begin{equation}\label{app1}
	        \bigcap_{m=1}^\infty E+\alpha_m\supseteq \bigcap_{m=1}^\infty K_m+\alpha_m.
	    \end{equation}
	    
	    But $K_1+\alpha_1=(F_1+\alpha_1)\cap (F_3+\alpha_1)\cap (F_5+\alpha_1)\cap \cdots$, $K_2+\alpha_2=(F_2+\alpha_2)\cap (F_6+\alpha_2)\cap (F_{10}+\alpha_2)\cap \cdots$, etc. 
	    We can rewrite the infinite intersection on the right hand side of \eqref{app1}  into:
	    \begin{equation}\label{app2}
	       \bigcap_{u=1}^\infty F_u+\alpha_{j_u},
	    \end{equation}
	    
		where $j_u$ is the greatest integer $v$ such that $2^{v-1}$ divides $u$. For example, the first few terms of the intersection are:
		$$
	    (F_1+\alpha_1)\cap (F_2+\alpha_2)\cap (F_3+\alpha_1)\cap(F_4+\alpha_3)\cap(F_5+\alpha_1)\cap (F_6+\alpha_2)\cap \cdots
		$$
		We would like to show this intersection is nonempty.
		
	    Denote $C_1:=[0, 1 /M_1]$. Since the distance between the centres of the two adjacent intervals in $F_2$ is $1 /(2M_1)$ and the intervals of $F_2$ are shorter in length than those of $F_1$, no matter how we translate $C_1$, there is an interval $C_2$ of $F_2$ that is contained in that translate of $C_1$.
		
		Hence for any given $\alpha_1,\alpha_2\in \mathbb R$, we can find such $C_2$ satisfying $C_2+\alpha_2\subseteq C_1+\alpha_1$. Similarly, one can find $C_3$ of $F_3$ such that $C_3+\alpha_1\subseteq C_2+\alpha_2$. Continuing in this way, we get a nested sequence of compact intervals with rapidly decreasing length:
		$$
		C_u+\alpha_{j_u}\subseteq C_{u-1}+\alpha_{j_{u-1}},\quad u\geq 2.
		$$
		By the nested interval theorem, the intersection in \eqref{app2} is nonempty, and hence so is the intersection in \eqref{app1}.
		\end{itemize}
	\end{proof}
	We remark that $E$ defined in this way is not closed. This was seen by taking $\{\alpha_m\}$ to be $\mathbb Q\cap [0,1]$ as in the introduction of the paper, but it can also be seen directly from this simplified construction. Indeed, $E^c$ is the set of all real numbers $x$ such that for any $j\in \mathbb N$, there is $k_j\in \mathbb N$ so that the $(2k-1)2^{j-1}$-th digit of $x$ is not $0$ or $M_n/2$. Particularly, if $x\in E^c$, then there is an increasing sequence $a_n\in \mathbb N$ such that the $a_n$-th digit of $x$ is not $0$ or $M_n/2$.
	
	If $E^c$ were open, this means for any $x\in E^c$, if $y$ is sufficiently close to $x$, then $y\in E^c$. However, we see that for any $\delta>0$, we can choose $|x-y|<\delta$ such that $y$ is a finite decimal number, so $y\notin E^c$.
	
	Lastly, we have that $E$ is dense in $\mathbb R$. Given $\epsilon>0$ and $x\in \mathbb R$, consider the digit expansion of $x$. There is some $j_0$ and some real number $y$ with the same digits as $x$ on all digits $1\leq j\leq j_0-1$ but having all digits $0$ for $j\geq j_0$, such that $|x-y|<\delta$. Then $y\in K_{j_0}\subseteq E$.

 \section*{Acknowledgement}
 I am grateful to my advisor Malabika Pramanik for her guidance throughout the preparation of this paper.

\end{document}